\documentclass[12pt,a4paper]{article}
\usepackage[dvips]{graphicx}
\usepackage{amssymb,latexsym}
\usepackage{mathrsfs}
\usepackage{psfrag}
\usepackage[all]{xy}
\usepackage{color}

\usepackage{amsmath}

\usepackage[ansinew]{inputenc}

\input epsf

\newtheorem{prop}{Proposition}[section]
\newtheorem{theo}[prop]{Theorem}
\newtheorem{cor}[prop]{Corollary}

\newtheorem{rem}[prop]{Remark}

\newtheorem{lem}[prop]{Lemma}

\newcommand{\R}{\mbox{\boldmath $\mathbb{R}$}}

\newenvironment{proof}
 {\begin{trivlist} \item[\hskip \labelsep {\bf Proof}\hspace*{3 mm}]}
 {\hfill$\Box$\end{trivlist}}

\begin{document}
\title{On the flat geometry of the cuspidal edge}
\author{Ra\'{u}l Oset Sinha\footnote{Supported by DGCYT and FEDER grant no. MTM2012-33073.}\, and\, Farid Tari
\footnote{Partially supported by the grants FAPESP 2014/00304-2, CNPq 301589/2012-7, 472796/2013-5.}}

\maketitle
\begin{abstract}
We study the geometry of the cuspidal edge $M$ in $\mathbb R^3$ derived from its contact with planes and lines (referred to as
flat geometry). The contact of $M$ with planes is measured by the singularities of the height functions on $M$.
We classify submersions on a model of $M$ by diffeomorphisms
and recover the contact of $M$ with planes from that classification.
The contact of $M$ with lines is measured by the
singularities of orthogonal projections of $M$. We list the generic singularities of the projections and
obtain the generic deformations of the apparent contour (profile) when the direction of projection varies locally in $S^2$.
We also relate the singularities of the height functions and of the projections
to some geometric invariants of the cuspidal edge.
\end{abstract}

\renewcommand{\thefootnote}{\fnsymbol{footnote}}
\footnote[0]{2010 Mathematics Subject classification 57R45, 53A05.}

\section{Introduction}\label{sec:intro}

Let $\phi:U\subset \mathbb R^2\to \mathbb R^3$ be a parametrisation of a surface $M$, where $U$ is an open set
and $\phi$ is an infinitely differentiable map.
The surface $M$ is called a {\it cuspidal edge} if it admits a parametrisation $\phi$ which is $\mathcal A$-equivalent to
$f(x,y)=(x,y^2,y^3)$, that is, there exist diffeomorphisms $h$ and $k$ such that $\phi=k\circ f\circ h^{-1}$.
Our study is local in nature so we consider germs of parametrisations of a cuspidal edge. Observe that the
cuspidal edge is singular along a curve and its trace on a plane transverse to this curve is a curve with a cusp singularity, see
Figure \ref{fig:Discvuk} (middle figure).

Cuspidal edges occur naturally in differential geometry. For instance, given a regular surface $M$ in $\mathbb R^3$,
one can consider its parallel $M_d$, which is the surface obtained by moving the points on $M$ along a chosen
unit normal vector to $M$ by a fixed distance $d$. The parallel $M_d$ can become singular and is, in general,
a cuspidal edge with its singularities corresponding to
points on the surface where $d=1/\kappa_i$, $i=1,2$, $\kappa_i$'s being the principal curvatures. (The singularities
of $M_d$ can become more degenerate than a cuspidal edge on some special curves on the surface $M$.)
Another example is the focal set (caustic) of a surface in $\mathbb R^3$. If we take a parametrisation where the lines of curvature
are the coordinate curves $x_i=constant$, $i=1,2$, then the focal set is a cuspidal edge at generic points on the curves $\partial \kappa_i/\partial x_i=0$, $i=1,2$, where $\kappa_i$ are the principal curvatures.

Because cuspidal edges occur naturally and in a stable way in some cases, it is of interest to study their
differential geometry. There is already work in this direction,
see for example \cite{bruce-wilkinson, martinsnuno, martinssaji, kentaro, teramoto, wilkinson}.

In this paper, we study the geometry of the cuspidal edge $M$ derived from its contact with planes and lines
(which is referred to as the flat geometry of $M$ as planes and lines are flat objects, i.e., have zero curvature).
Consider parallel planes orthogonal to  ${\bf v}\in S^2$, where  $S^2$ denotes the unit sphere in $\mathbb R^3$.
These planes are the fibres of the function $h_{{\bf v}}(p)=p\cdot{} {\bf v}$, where $``\cdot{}"$ is the scalar product in $\mathbb R^3$.
The contact of $M$ with the above planes at $p_0=\phi(0,0)$ is measured by the singularities
of $h_{{\bf v}}\circ \phi$ at the origin.
By varying ${\bf v}$, we get the family of
height functions
$H:U\times S^2\to \mathbb R$ on $M$ given by
$$H((x,y),{\bf v})=\phi(x,y)\cdot{}{\bf v}.$$

In the above setting the model (flat) surfaces in $\mathbb R^3$ are planes and the parametrisation $\phi$ is taken in general form.
In this paper, we follow the approach in \cite{brucewest} and invert the situation: we fix the $\mathcal A$-model $X$ of the cuspidal edge
as the image of the map-germ $f(x,y)=(x,y^2,y^3)$ and consider its contact with fibres of submersions.

We classify in \S \ref{sec:submersions} submersion $\mathbb R^3,0\to \mathbb R$ up to changes of coordinates
in the source that preserve the model cuspidal edge $X$. Such changes of coordinates form a geometric subgroup $\mathcal R(X)$  of the Mather group $\mathcal R$
(see \cite{damon}). In \S \ref{sec:Geomsubmersions}, we deduce from that classification the generic geometry of the contact of a cuspidal edge $M$ with planes. We study the duals of these generic cases. Other results on duals of cuspidal edges can be found in \cite{teramoto}.

The contact of $M$ with lines is measured by the $\mathcal A$-singularities of orthogonal projections of $M$ to planes. Here too we fix the model cuspidal edge $X$ and
classify  in \S \ref{sec:projections}  the singularities of germs of submersions $\mathbb R^3,0\to \mathbb R^2,0$ under the action of the subgroup $_X\mathcal A=\mathcal R(X)\times \mathcal L$
of the Mather group $\mathcal A$.
This approach has an important advantage to considering the $\mathcal A$-singularities of the orthogonal
projections on $M$ (or the $\mathcal R$-singularities of the height functions on $M$). Using a tansversality result from \cite{brucewest} adapted to our situation,
we can state that only the singularities of $_X\mathcal A_e$-codimension $\le 2$ can occur for a generic cuspidal edge $M$. Furthermore, we associate a
natural 2-parameter family of submersions $\mathbb R^3,0\to \mathbb R^2,0$ on the model $X$ obtained from the family of orthogonal projections of $M$ to planes.
This family is an $_X\mathcal A_e$-versal unfolding of the generic singularities of the submersions on $X$.
This allows us to obtain in \S \ref{sec:Geometryprojections}  the generic deformations of the apparent contour (profile) of $M$
when the direction of projection varies locally in $S^2$.

\section{Preliminaries} \label{sec:prel}

We review in this section some aspects of the geometry of the cuspidal edge (\S\ref{ssec:geomcuspidalede}) and establish
some notation (\S\ref{ssec:notaSing}) for the classification of germs of functions and mappings on the cuspidal edge.

\subsection{Geometric cuspidal edge}\label{ssec:geomcuspidalede}

Let $M$ be a general cuspidal edge in $\mathbb R^3$ which we shall refer to (following the notation in \cite{brucewest}) as
a geometric cuspidal edge. In \cite{martinssaji} a local parametrisation (at the origin) of the cuspidal edge is given by
allowing any changes of coordinates in the source and changes of coordinates in the target given by isometries.
The parametrisation, which we shall adopt in the rest of the paper,
 is the following
\begin{equation}\label{eq:prenormalform}
\phi(x,y)=(x,a(x)+\frac{1}{2}y^2,b_1(x)+y^2b_2(x)+y^3b_3(x,y)),
\end{equation}
with $(x,y)$ in a neighbourhood of the origin and
$a(0)=a'(0)=0$,  $b_1(0)=b_1'(0)=0$, $b_2(0)=0$, $b_3(0)\ne 0$.
Following the notation in \cite{martinssaji}, we write
$$
\begin{array}{rcl}
a(x)&=&\frac{1}{2}a_{20}x^2+\frac{1}{6}a_{30}x^3+\frac{1}{24}a_{40}x^4+O(5),\\
b_1(x)&=&\frac{1}{2}b_{20}x^2+\frac{1}{6}b_{30}x^3+\frac{1}{24}b_{40}x^4+O(5),\\
b_2(x)&=&\frac{1}{2}b_{12}x+\frac{1}{6}b_{22}x^2+O(3),\\
b_3(x,y)&=&\frac{1}{6}b_{03}+ \frac{1}{6}b_{13}x+O(2).
\end{array}
$$

The tangential direction of $M$ at the origin is along $(1,0,0)$ and
its tangent cone is the plane $w=0$, where $(u,v,w)$ are the coordinates of $\mathbb R^3$.
The singular set $\Sigma$ of $M$ is the image of the line $y=0$ and
is parametrised by
$
\alpha(x)=\phi(x,0)=(x,a(x),b_1(x)).
$

If we denote by $\kappa_{\Sigma}$ and $\tau_{\Sigma}$ the curvature
and the torsion of $\Sigma$ as a space curve, then
$$\begin{array}{l}
\kappa_{\Sigma}(0)=\sqrt{a_{20}^2+b_{20}^2}, \\
 \tau_{\Sigma}(0)=\frac{a_{20}b_{30}-b_{20}a_{30}}{a_{20}^2+b_{20}^2}, \\
\tau_{\Sigma}'(0)=\frac{1}{a_{20}^2+b_{20}^2}(
a_{20}b_{40}-b_{20}a_{40}-\frac{2}{a_{20}^2+b_{20}^2}(a_{20}b_{30}-b_{20}a_{30})(a_{20}a_{30}+b_{20}b_{30})
).
\end{array}
$$

The osculating plane of $\Sigma $ at the origin is orthogonal to the vector $(0,-b_{20},a_{20})$.
It coincides with the tangent cone to $M$ at the origin if and only if $b_{20}=0$.
It is worth observing that this happens if and only if the closure of the parabolic curve of the regular part of $M$ intersects
the singular set $\Sigma$.

Recall that a smooth curve has contact of type $A_{\ge 1}$ at a point $p$ with any of its tangent planes at $p$, of type $A_{\ge 2}$ if the
plane is the osculating plane at $p$ and of type  $A_{ 3}$ if furthermore the torsion of the curve vanishes at $p$ but the derivative of the torsion is not zero at that point
(see for example \cite{brucegiblinBook}). In \cite{kentaro} other invariants of the cuspidal edge are defined. These are:

The singular curvature $\kappa_s$  ($\kappa_s(0)=a_{20}$);

The limiting normal curvature $\kappa_n$ ($\kappa_n(0)=b_{20}$);

The cuspidal curvature $\kappa_c$  ($\kappa_c(0)=b_{03}$);

The cusp-directional torsion $\kappa_t$  ($\kappa_t(0)=b_{12}$);

The edge inflectional curvature $\kappa_i$  ($\kappa_i(0)=b_{30}$).

\medskip
The contact of $M$ with lines and planes is affine
invariant (\cite{bgt95}), so we can allow affine changes of coordinates in the target without changing the type of contact.

Given a parametrisation $\psi:\mathbb R^2,0\to \mathbb R^3,0$, we can make a rotation in the target
and changes of coordinates in the source and write
$j^2\psi(x,y)=(x,Q_1(x,y),Q_2(x,y))$, with $Q_1,Q_2$ homogeneous polynomials of degree 2.
We can consider the
 ${\cal G}=GL(2,{\mathbb R})\times GL(2,{\mathbb R})$-action on the set of pairs of quadratic forms
 $(Q_1,Q_2)$.
Following \cite{martinssaji}, we can set $(Q_1,Q_2)=(\frac{a_{20}}{2}x^2+\frac{1}{2}y^2,\frac{b_{20}}{2}x^2)$
 by isometric changes of coordinates in the target and any smooth changes of coordinates in the source,  and this is ${\cal G}$-equivalent to

\begin{center}
\begin{tabular}{cl}
$(y^2,x^2)$ & if and only of $b_{20}\ne 0$ (hyperbolic)\cr
$(\pm x^2+y^2,0)$& if and only of $b_{20}=0$, $a_{20}\ne 0$  (inflection)\cr
$(y^2,0)$& if and only of $b_{20}=a_{20}= 0$ (degenerate inflection)
\end{tabular}
\end{center}

The above $\cal G$-classes are the only ones that can occur for the pair
$(Q_1,Q_2)$ associated to a parametrisation of a cuspidal edge.
In particular, most points on $\Sigma$ are hyperbolic points for the pair $(Q_1,Q_2)$ and
we have an inflection point if and only if the osculating plane of $\Sigma $ coincides with the tangent cone to $M$.

Following the above discussion, we can take $(Q_1,Q_2)$ associated to $\phi$ in \eqref{eq:prenormalform}
in one of the $\cal G$ normal forms above. However we shall work with the parametrisation \eqref{eq:prenormalform}
to make the interpretation of the conditions we get match some of the invariants in \cite{martinssaji}.

\subsection{Classification tools}\label{ssec:notaSing}

Let ${\cal E}_n$ be the local ring of germs of functions
$\mathbb{R}^n,0 \to \mathbb{R}$ and ${\mathcal M}_n$  its maximal
ideal.
Denote by ${\cal E}(n,p)$ the $p$-tuples of elements in ${\cal
E}_n$.  Let ${\cal A}={\cal R}\times{\cal
L}$
denote the group
of pairs of germs of diffeomorphisms of the source and target, which acts smoothly on ${\mathcal
M}_n.{\cal E}(n,p)$ by $(k_1,k_2).G=k_2\circ G\circ k_1^{-1}$.
The tangent space to the $\mathcal{A}$-orbit of $F$ at the germ $F$
is given by
 $$
 L{\cal A}{\cdot}{F}={\mathcal M}_n.\{F_{x_1},\ldots,F_{x_n}\}+F^*({\mathcal M}_p).\{e_1,\ldots,e_p\},
 $$
where $F_{x_i}$ are the partial derivatives of $F$ with respect to $x_i$
($i=1,\ldots,n$), $e_1,\ldots,e_p$ denote the standard basis vectors
of ${\mathbb R}^p$ considered as elements of ${\cal E}(n,p)$, and
$F^*({\mathcal M}_p)$ is the pull-back of the maximal ideal in
${\cal E}_p$. The extended tangent space to the $\mathcal{A}$-orbit
of $F$ at the germ $F$ is given by
$$
L_e{\cal A}{\cdot}{F}={\cal E}_n.\{F_{x_1},\ldots,F_{x_n}\}+F^*({\cal E}_p).\{e_1,\ldots,e_p\},
$$
and the codimension of the extended orbit is $
d_e(F,{\mathcal{A}})=\dim_{\mathbb{R}}({\cal
E}(n,p)/L_e{\mathcal{A}}(F))\,. $

Let $k \geq 1$ be an integer. We denote by $J^k(n,p)$ the space of
$k$th order Taylor expansions without constant terms of elements of
${\cal E}(n,p)$
 and write $j^kF$ for the $k$-jet of $F$. A germ $F$ is said to be {\it
$k-\mathcal{A}$-determined} if any $G$ with $j^kG=j^kF$ is
$\mathcal{A}$-equivalent to $F$ (notation: $G\sim_{\mathcal A} F$). The $k$-jet
of $F$ is then called a sufficient jet.

Our goal in \S \ref{sec:submersions} and \S \ref{sec:projections}
is to classify germs of functions and mappings on $X \subset \R^3$, where $X$ is the germ of
the smooth model of a cuspidal edge.
This means that we require that the
diffeomorphisms in $\mathbb R^3$  preserve $X$. We
follow the method in \cite{brucewest} and recall some results from
there. Let $X,0\subset \mathbb R^n,0$ be the germ of a reduced
analytic sub-variety of $\mathbb R^n$ at $0$ defined by a polynomial
$h$ in $\mathbb R[x_1,\ldots,x_n]$. Following \mbox{Definition 3.1} in
\cite{brucewest}, a diffeomorphism $k:\mathbb R^n,0\rightarrow
\mathbb R^n,0$ is said to preserve $X$ if $k(X),0=X,0$ (i.e.,
$k(X)$ and $X$ are equal as germs at $0$). The group of such
diffeomorphisms is a subgroup of the group $\mathcal{R}$ and is
denoted by $\mathcal{R}(X)$ ($_{ X}\mathcal R$ in some texts).
We denote by $_X\mathcal A=\mathcal R(X)\times\mathcal L$ the subgroup $\mathcal A$ where the diffeomorphisms in the source
preserve $X$.


\section{Functions on a cuspidal edge}\label{sec:submersions}

Given the $\mathcal A$-normal form
$f(x,y)=(x,y^2,y^3)$ of a cuspidal edge, we classify germs of
submersions $g:\mathbb R^3,0\to \mathbb R,0$ up to $\mathcal
R(X)$-equivalence, with $X=f(\mathbb R^2,0)$.
The defining equation of $X$ is given by $h(u,v,w)=v^3-w^2$.

Let $\Theta(X)$ be the $\mathcal{E}_3$-module of vector fields in
$\mathbb R^3$ tangent to $X$.
We have $\xi\in\Theta(X)$ if and only if $\xi h=\lambda h$
for some function $\lambda$ (\cite{brucewest}).

\begin{prop}\label{prop:genO(X)}
The $\mathcal{E}_3$-module $\Theta(X) $ of vector fields in
$\mathbb R^3$ tangent to $X$ is generated by the vector fields
$
\xi_1=\frac{\partial}{\partial u},$
$\xi_2=2v\frac{\partial}{\partial v}+3w\frac{\partial}{\partial w},$
$\xi_3=2w\frac{\partial}{\partial v}+3v^2\frac{\partial}{\partial w}.
$
\end{prop}

\begin{proof}
The defining equation $h(u,v,w)=v^3-w^2$ of $X$ is weighted
homogenous in $v$ and $w$ with weights 2 and 3 respectively. The result follows by applying
\mbox{Proposition 7.2} in \cite{bruceroberts} for isolated singularities to the cusp $v^3-w^2=0$ in the $(v,w)$-plane and adding the trivial vector field $\xi_1$ in $\mathbb R^3$.
\end{proof}

Let $\Theta_1(X)=\{\delta\in\Theta(X):j^1\delta=0\}$. It follows from Proposition \ref{prop:genO(X)} that
$$
\Theta_1(X)=\mathcal M_3^2.\{ \mathcal \xi_1\}+\mathcal M_3.\{ \mathcal \xi_2, \xi_3\}.
$$

For $f\in \mathcal E_3$, we define  $\Theta(X){\cdot} f=\{\eta(f)\,|\,
\eta\in \Theta(X)\}$. We define similarly $\Theta_1(X){\cdot} f$ and
the following tangent spaces to the  $\mathcal R(X)$-orbit of $f$ at
the germ $f$:
$$
L\mathcal R_1(X){\cdot}f=\Theta_1(X){\cdot} f,\quad
L\mathcal R(X){\cdot}f=L_e{\cal R}(X){\cdot}f=\Theta(X){\cdot} f.
$$

The ${\mathcal{R}(X)}$-codimension of $f$  is given by
$d(f,{\mathcal{R}(X)})=\dim_{\mathbb{R}}({\cal M}_3/L{\mathcal{R}(X)}(f))\,.$

The classification (i.e., the listing of representatives of the
orbits) of $\mathcal R(X)$-finitely determined germs is carried out
inductively on the jet level. The method used here is that of the
complete transversal \cite{bkd} adapted for the $\mathcal
R(X)$-action.
We have the following result which is a version of
Theorem 3.11 in \cite{brucewest} for the group $\mathcal R(X)$.

\begin{prop}\label{prop:completeTrans}
Let $f:\mathbb R^3,0\to \mathbb R,0$ be a smooth germ and
$h_1,\ldots, h_r$ be homogeneous polynomials of degree $k+1$ with
the property that
$$
\mathcal{M}_3^{k+1} \subset L\mathcal R_1(X){\cdot}f + sp\{
h_1,\ldots, h_r\} +\mathcal{M}_3^{k+2}.
$$
Then any germ $g$ with $j^kg(0)=j^kf(0)$ is $\mathcal
R_1(X)$-equivalent to a germ of the form $f(x)+\sum_{i=1}^l
u_ih_i(x)+\phi(x) $, where $\phi(x)\in \mathcal M_n^{k+2}$. The
vector subspace $sp\{ h_1,\ldots, h_r\}$ is called a complete
$(k+1)$-$\mathcal R(X)$-transversal of $f$.
\end{prop}

\begin{cor}\label{cor:detrminacy} If $ \mathcal{M}_3^{k+1}\subset L\mathcal R_1(X){\cdot}f
+\mathcal{M}_3^{k+2}$ then $f$ is $k-\mathcal{R}(X)$-determined.
\end{cor}

We also need the following  result about trivial families.

 \begin{prop}{\rm (\cite{brucewest})} \label{prop:trivialfam}
Let $F:\mathbb R^3\times \mathbb R,(0,0)\to \mathbb R,0$ be a smooth
family of functions with $F(0,t)=0$ for $t$ small. Let
$\xi_1,\ldots,\xi_p$ be vector fields in $\Theta(X)$ vanishing at
$0\in \mathbb R^n$. Then the family $F$ is $k-\mathcal R(X)$-trivial if
$\frac{\partial F}{\partial t}\in \left\langle  \xi_1(F),\ldots,\xi_p(F) \right\rangle +\mathcal M^{k+1}_n.
$
\end{prop}

Two families of germs of functions $F$ and $G:(\mathbb R^3 \times \mathbb R^l,(0,0))\to (\mathbb{R},0)$  are
\mbox{$P$-$\mathcal{R}^+(X)$}-equivalent if
there exist a germ of a diffeomorphism
$\Phi:(\mathbb R^3 \times \mathbb R^l,(0,0)) \to (\mathbb R^3 \times \mathbb R^l,(0,0))$
preserving $(X\times \mathbb R^l,(0,0))$ and of the form
$\Phi(x,u)=(\alpha(x,u),\psi(u))$ and a germ of a function $c:(\mathbb R^l,0) \to\mathbb R$
such that
$
G(x,u) = F(\Phi(x,u)) + c(u).
$

A family $F$ is said to be an $\mathcal{R}^+(X)$-versal deformation of $F_0(x)=F(x,0)$ if any other deformation $G$ of $F_0$
can be written in the form  $G(x,u) = F(\Phi(x,u)) + c(u)$ for some germs of
smooth mappings $\Phi$ and $c$ as above with $\Phi$
not necessarily a germ of diffeomorphism.

Given a family of germs of functions $F$, we write
$\dot{F}_i(x)=\frac{\partial F}{\partial u_i}(x,0).$

\begin{prop}\label{theo:InfcondUnfFunc}
A deformation $F:(\mathbb R^3 \times \mathbb R^l,(0,0))\to (\mathbb{R},0)$ of a germ of a function $f$
on $X$ is
$\mathcal{R}^+(X)$-versal if and only if
$$
L\mathcal R_e(X)\cdot{}f + \mathbb R.\left\{1, \dot{F}_1,\ldots,\dot{F}_l\right\}= \mathcal{E}_n.
$$
\end{prop}

We can now state the result about the  $\mathcal R(X)$-classification of germs of submersions.

\begin{theo} \label{theo:Classification}
Let $X$ be the germ of the $\mathcal A$-model of the cuspidal edge parametrised by
$f(x,y,z)=(x,y^2,y^3)$. Denote by $(u,v,w)$ the coordinates in the target.
Then any $\mathcal R(X)$-finitely determined germ of a submersion in $\mathcal M_3$
with $\mathcal R(X)$-codimension $\le 2$
(of the stratum in the presence of moduli)  is
$\mathcal R(X)$-equivalent to one of the germs in {\rm Table \ref{tab:germsubm}}.

\begin{table}[ht]
\caption{Germs of submersions in $\mathcal M_3$ of $\mathcal R(X)$-codimension $\le 2$.}
\begin{center}
{\begin{tabular}{lcl}
\hline
Normal form & $d(f,\mathcal{R}(X))$ &$\mathcal{R}^+(X)$-versal deformation\\
\hline
$u$ & $0$ &$u$\\
$\pm v\pm u^2$ & $0$&$\pm v\pm u^2$\\
$\pm v+ u^3$& $1$&$\pm v+u^3+a_1u$\\
$\pm v\pm u^4$& $2$&$\pm v\pm u^4+a_{2}u^{2}+a_1u$\\
$w+u^2$ &$1$&$w+u^2+a_1v$\\
$w+ uv+au^3$, $a\ne 0,-\frac{4}{27}$ &$2^{(*)}$&
 $w+uv+au^2+a_2u^2+a_1u$\\
\hline
\end{tabular}
}
\\
{\footnotesize
$(*)$: $a$ is a modulus and the codimension is that of the stratum.}
\end{center}
\label{tab:germsubm}
\end{table}

\end{theo}

\begin{proof} To simplify notation, we write complete $k$-transversal for
complete $k-\mathcal R(X)$-transversal, equivalence for $\mathcal
R(X)$-equivalence and codimension for $\mathcal R(X)$-codimension. In all the proof, $\xi_1, \xi_2,\xi_3$ are as in Proposition \ref{prop:genO(X)}.

The linear changes of coordinates in $\mathcal R(X)$ obtained by
integrating the 1-jets of vector fields in $\Theta(X)$ are
$$
\begin{array}{rcl}
\eta_1(u,v,w)&=&(\alpha u+\beta v+\gamma w,v,w), \alpha \ne 0\\
\eta_2(u,v,w)&=&(u,e^{2\alpha}v,e^{3\alpha}w), \alpha\in \mathbb R\\
\eta_3(u,v,w)&=&(u,v+\alpha w,w)
\end{array}
$$

Consider a non-zero 1-jet $g=au+bv+cw$. If $a\ne 0$, then $g$ is equivalent to $u$ and
its codimension is $0$.

Suppose that $a=0$. If $b\neq 0$, we use $\eta_3$ to set $c=0$ and $\eta_2$ to set $b=\pm 1$. If $a=b=0$ but $c\ne 0$ we can set
$c=\pm 1$. Observe that $(u,v,w)\mapsto (u,v,-w)$ preserves $X$, so we can set $c=1$. Thus, the orbits of submersions in
the 1-jet space are
$
u,\pm v,w.
$

$\bullet$ Consider the 1-jet $g=v$ (the results follow similarly for $g=-v$). Then $\xi_1(g)=0$, $\xi_2(g)=2v$
and $\xi_3(g)=2w$, so for any integer $k\ge 2$, $\mathcal M_{3}^k\subset L\mathcal R_1(X)\cdot{}g + sp\{u^k\} + \mathcal M_{3}^{k+1}$, that is, a
complete $k$-transversal is given by $g=v+\lambda u^k$. Using
$\eta_2$ we can set $\lambda=\pm 1$ if $\lambda\ne 0$ (if $k$ is odd we can set $\lambda =1$). For the germ $g=v\pm u^k$,
 we have
$
\xi_1(g)=\pm ku^{k-1},
\xi_2(g)=2v,
\xi_3(g)=2w.
$ Now $\mathcal M_{3}^{k+1}\subset L\mathcal R_1(X)\cdot{}g + \mathcal
M_{3}^{k+2}$, so $v\pm u^k$ is $k$-determined. Its codimension is
$k-2$, and clearly, $ v\pm u^k+a_{k-2}u^{k-2}+\ldots+a_1u_1$ is an $\mathcal R^+(X)$-versal
deformation.

\medskip
$\bullet$ We consider now the 1-jet $g=w$. We have $\xi_1(g)=0$,
$\xi_2(g)=3w$ and $\xi_3(g)=3v^2$.

A complete $2$-transversal is given by $g=w+\lambda_1u^2+\lambda_2uv+\lambda_3v^2$.
We can consider $g$ as a 1-parameter family of germs of functions
parametrised by $\lambda_3$. Then
$\partial g/\partial {\lambda_3}=v^2\in \langle \mathcal M_3
\xi_1(g),\xi_2(g),\xi_3(g) \rangle + \mathcal M_3^{3}$, so by
Proposition \ref{prop:trivialfam}, $g$ is equivalent to
$w+\lambda_1u^2+\lambda_2uv$.
We proceed similarly to show that $g$ is
trivial along $\lambda_2$ if $\lambda_1\ne 0$.
Thus, if $\lambda_1\ne 0$ we can take
$g=w+\lambda_1 u^2$. We can now set $\lambda_1=\pm
1$ by rescaling. For the germ $g=w\pm u^2$, we have
$
\xi_1(g)=\pm 2u,
\xi_2(g)=3w,
\xi_3(g)=3v^2,
$
so $\mathcal M_{3}^{3}\subset L\mathcal R_1(X)\cdot{}g + \mathcal
M_{3}^{4}$, that is, $g$ is $2$-determined.
An $\mathcal R^+(X)$-versal deformation is given $g=w\pm u^2+a_1v$ and the germ has codimension 1.

If $\lambda_1=0$ but $\lambda_2\ne 0$ we can set $\lambda_2=1$ and consider $g=w+uv$. Then
$\xi_1(g)=v$, $\xi_2(g)=2uv+3w$ and $\xi_3(g)=2wu+3v^2$, so $\mathcal M_{3}^3\subset L\mathcal
R_1(X)\cdot{}g + sp\{u^3\} + \mathcal M_{3}^{4}$, that is, a
complete $3$-transversal is given by $g=w+uv+a u^3$. Here $a$ is a parameter modulus.

For $g=w+uv+au^3$ we have
$\xi_1(g)=v+3au^{2},$ $\xi_2(g)=2uv+3w,$
$\xi_3(g)=2uw+3v^2.$
Using the vectors
$
u^2\xi_2(g)=2u^3v+3u^2w,
u\xi_3(g)=2u^2w+3uv^2,
uv\xi_1(g)=uv^2+3avu^{3}
$
which are in $\mathcal R_1(X)\cdot{}g$, we get
$u^2w$, $uv^2$ and $u^3v$ if $a\ne -\frac{4}{27}$. Then, using
$u^3\xi_1(g)$ we get $u^5$ if $a\ne 0$. Now
using $\xi_1$ and $\xi_2$ we get all monomials divisible by $v$ and $w$ of degree $5$ in $L\mathcal R_1(X)\cdot{}g + \mathcal M_{3}^{6}$. Therefore $g$ is $5$-determined
if $a\ne 0, -\frac{4}{27}$.
A complete $4$-transversal of $g$ is
$g+\lambda u^{4}$. Using Mather's Lemma (the vectors
$\xi_3(g)$, $v\xi_1(g)$, $u\xi_2(g)$ give $uw,$ $v^2,$ $u^2v$ if
$a\ne -\frac{4}{27}$, then  $u^2\xi_1(g)$ gives $u^4$ if $a\ne 0$)
we show that $g+\lambda u^{4}$ is $\mathcal R(X)$-equivalent to $g$. Therefore $g$ is $4$-determined.
An $\mathcal R^+(X)$-versal deformation is given $g=w+uv+a u^3+a_{2}u^{2}
+a_1u$ and has codimension $3$ (the codimension of the stratum is $2$).
\end{proof}

\subsection{The geometry of functions on a cuspidal edge}\label{sec:Geomsubmersions}

We consider the $\mathcal A$-model cuspidal edge $X$ parametrised by $f(x,y)=(x,y^2,y^3)$ and with equation $v^3-w^2=0$.
The tangential line at a singular point is parallel to $(1,0,0)$ and the tangent cone to $X$ at a singular point is the plane $w=0$.

Given a deformation $F:\mathbb R^3\times \mathbb R^2,0\to \mathbb R$ of a germ $g$ on $X$, we consider the family
$G(x,y,a)=F(f(x,y),a)$ and the following sets:
$$
\mathscr D_1(F)=\{(a,G(x,y,a))\in \mathbb R^2\times \mathbb R:
\frac{\partial G}{\partial x}=\frac{\partial G}{\partial y}=0 \mbox{ at } (x,y,a)
\}
$$
and
$$
\mathscr D_2(F)=\{(a,G(x,0,a))\in \mathbb R^2\times \mathbb R: \frac{\partial G}{\partial x}=0 \mbox{ at  } (x,0,a)
\}.
$$

It is not difficult to show that for  two $P$-$\mathcal R^+(X)$-equivalent deformations $F_1$ and $F_2$
the sets $\mathscr D_1(F_1)$ and $\mathscr D_1(F_2)$ are diffeomorphic and so are $\mathscr D_2(F_1)$ and $\mathscr D_2(F_2)$.
Therefore, it is enough to compute the sets $\mathscr D_1(F)$ and $\mathscr D_2(F)$
for the deformations in Table \ref{tab:germsubm}.

\medskip
$\bullet$ {\it The  germ $g=u$.}\\
The fibre $g=0$ is a plane transverse to both the tangential line and to the tangent cone to $X$.
Here an $\mathcal R^+(X)$-versal deformation  is $F(u,v,w,a_1,a_2)=u$ and
both $\mathscr D_1(F)$ and $\mathscr D_2(F)$ are the empty set.

\medskip
$\bullet$ {\it The germs  $g=\pm v\pm u^k$, $k=2,3,4$.}

The fibre $g=0$ is tangent to the tangential line of $X$ at the origin but is transverse
the tangent cone to $X$.
The contact of the tangential line with the fibre $g=0$ is measured by the singularities
of $g(x,0,0)=\pm x^k$, so it is of type $A_{k-1}$.

\smallskip
(i) $k=2$. Here an $\mathcal R^+(X)$-versal deformation  is $F(u,v,w,a_1,a_2)=\pm v\pm u^2$.
Then $G(x,y,a_1,a_2)=\pm y^2\pm x^2$, and both $\mathscr D_1(F)$ and $\mathscr D_2(F)$ are planes (Figure \ref{fig:Discvuk}, left).

\smallskip
(ii) $k=3$. We have $F(u,v,a_1,a_2)=\pm v+u^3+a_1u$ and $G(x,y,a)=\pm y^2+ x^3+a_1x$. Thus
$\frac{\partial G}{\partial x}=\frac{\partial G}{\partial y}=0$ when $y=0$ and $a_1=-3x^2$. The set $\mathscr D_1(F)$
is a surface parametrised by $(x,a_2)\mapsto (-3x^2,a_2,-2x^3)$, i.e.,  is a cuspidal edge.
The set $\mathscr D_2(F)$ is also a cuspidal edge and coincides with $\mathscr D_1(F)$ (Figure \ref{fig:Discvuk}, middle).

\smallskip
(iii)  $k=4$. Here $F(u,v,w,a_1,a_2)=\pm v\pm u^4+a_2u^2+a_1u$ and $G(x,y,a_1,a_2)=\pm y^2\pm  x^4+a_2x^2+a_1x$, so
$\frac{\partial G}{\partial x}=\frac{\partial G}{\partial y}=0$ when $y=0$ and $a_1=\mp 4x^3-2a_2x$.
 The set $\mathscr D_1(F)$
is a surface parametrised by $(x,a_2)\mapsto (\mp 4x^3-2a_2x,a_2,\mp 3x^4-a_2x^2)$, which is
swallowtail surface.
The set $\mathscr D_2(F)$ is also a swallowtail surface and coincides with $\mathscr D_1(F)$ (Figure \ref{fig:Discvuk}, right).

\begin{figure}
\begin{center}
\includegraphics[width=10cm, height=2.5cm]{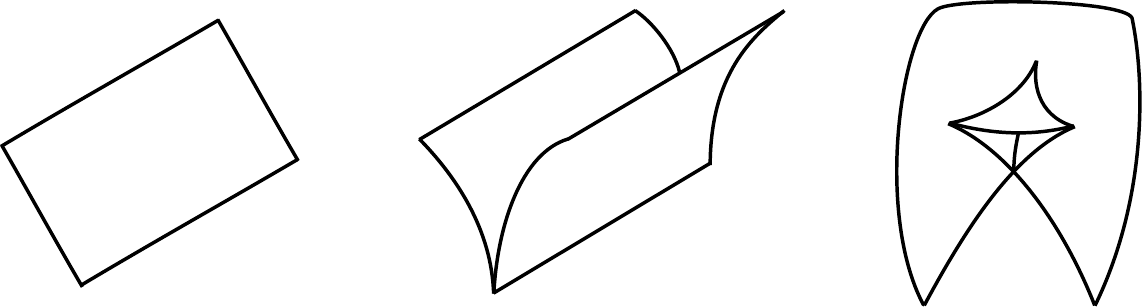}
\caption{Discriminant $\mathscr D_1(F)$ of versal deformations of $\pm v\pm u^k$, from left to right: $k=2,3,4$.
Here $\mathscr D_2(F)$ coincides with $\mathscr D_1(F)$.}
\label{fig:Discvuk}
\end{center}
\end{figure}

\medskip
$\bullet$ {\it The germ $g=w\pm u^2$.}

The tangent plane to the fibre $g=0$ coincides with the tangent cone to $X$ at the origin (and contains the tangential direction to $X$ at that point). The contact of the fibre $g=0$ with the tangential line is an ordinary one (of type $A_1$) as
$g(x,0,0)=\pm x^2$.

We have $F(u,v,a_1,a_2)=w+u^2+a_1v$ and $G(x,y,a_1,a_2)= x^2+y^3+a_1y^2$, so
$\frac{\partial G}{\partial x}=\frac{\partial G}{\partial y}=0$ when $x=0$ and $y(3y+a_1)=0$.
When $y=0$ we get the plane $(a_1,a_2,0)$ and for $3y+a_1=0$ we get a surface
parametrised by $(x,a_2)\mapsto (-3y,a_2,-2y^3)$. The set $\mathscr D_1(F)$ is the union of these two surfaces which have an $A_2$-contact along the $a_2$-axis, see Figure \ref{fig:Discwu2} (left).
The set $\mathscr D_2(F)$ is the plane $(a_1,a_2,0)$.

\begin{figure}
\begin{center}
\includegraphics[width=5cm, height=3.5cm]{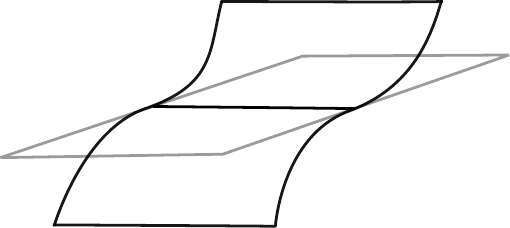}
\includegraphics[width=5cm, height=3.5cm]{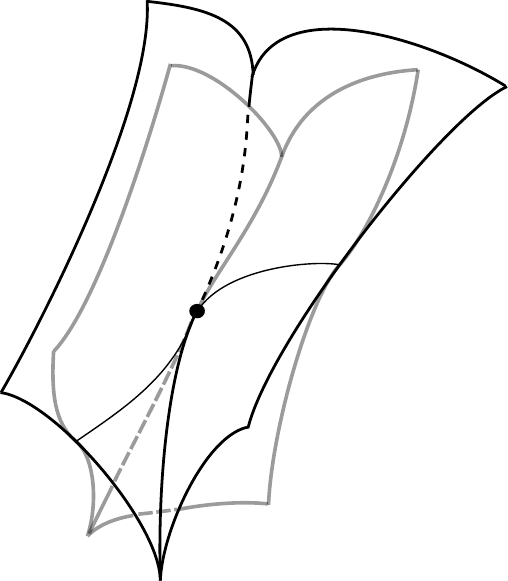}
\caption{Left: discriminant $\mathscr D_1(F)$ of a versal deformation of $w\pm u^2$,
which is the union of two smooth surfaces having an $A_2$-contact along a curve.
Right: Discriminant $\mathscr D_1(F)$ of a versal deformation of $w+uv+au^3$, which is the union of two cuspidal edges. In both figures
the discriminant $\mathscr D_2(F)$, which is a subset of $\mathscr D_1(F)$, is the surface in gray.}
\label{fig:Discwu2}
\end{center}
\end{figure}

\medskip
$\bullet$ {\it  The germ $g=w+uv+au^3$, $a\ne 0$.}

Here too, as in the previous case, the tangent plane to the fibre $g=0$ coincides with the tangent cone to $X$ at the origin.
However, the contact of the fibre $g=0$ with the tangential line is of order $3$
as  $g(x,0,0)=ax^3$.

We have $F(u,v,w,a_1,a_2)=w+uv+au^3+a_2u^2+a_1u$ and  $G(x,y,a_1,a_2)= xy^2+y^3+ax^3+a_2x^2+a_1x$.
Differentiating we get
$$\begin{array}{l}
\frac{\partial G}{\partial x}=y^2+3ax^2+2a_2x+a_1\\
\frac{\partial G}{\partial y}=2xy+3y^2.
\end{array}
$$
We have $\frac{\partial G}{\partial y}=0$ when $y=0$ or $y=-(2/3)x.$ Substituting in
$\frac{\partial G}{\partial x}=0$ gives $\mathscr D_1(F)$
as the union of two surfaces parametrised by
$$(a_2,x)\mapsto (-3ax^2-2a_2x,a_2,-2ax^3-a_2x^2)$$
and
$$(a_2,x)\mapsto (-(\frac{4}{9}+3a)x^2-2a_2x,a_2,-(\frac{8}{27}+2a)x^3-a_2x^2).$$
Both these surfaces are cuspidal edges
and are as in Figure \ref{fig:Discwu2} (right).
The set $\mathscr D_2(F)$ coincides with the first cuspidal edge.


\subsection{Contact of a geometric cuspidal edge with planes}\label{sec:Contactplanes}

The family of height functions $H:M\times S^2\to \mathbb R$ on $M$
is given by $H((x,y),{\bf v})=H_{{\bf v}}(x,y)=\phi(x,y)\cdot {\bf v}.$
The height function $H_{{\bf v}}$ on $M$ along a fixed direction ${{\bf v}}$ measures the contact of $M$ at $p$
with the plane $\pi_{{\bf v}}$ through $p$ and orthogonal to ${{\bf v}}$.
The contact of $M$ with  $\pi_{{\bf v}}$  is described by that of the fibre
$g=0$ with the model cuspidal edge $X$, with $g$ as in Theorem \ref{theo:Classification}. Following the transversality theorem in the Appendix of \cite{brucewest}, for a generic cuspidal-edge, the height functions $H_{{\bf v}}$, for any ${{\bf v}}\in S^2$,
can only have singularities of $\mathcal R(X)$-codimension $\le 2$ (of the stratum) at any point on the cuspidal edge.

We shall take $M$ parametrised as in \eqref{eq:prenormalform}
and write ${\bf v}=(v_1,v_2,v_3)$.
Then,
$$
H_{{\bf v}}(x,y)=H((x,y),{\bf v})=xv_1+(a(x)+\frac{1}{2}y^2)v_2+(b_1(x)+y^2b_2(x)+y^3b_3(x,y))v_3.
$$

The function $H_{{\bf v}}$  is singular at the origin
if and only if $v_1=0$, that is, if and only if the plane $\pi_{{\bf v}}$ contains the tangential direction to $M$ at the origin.

When $\pi_{{\bf v}}$ is transverse to the tangential direction to $M$ at the origin,
the contact of $M$ with $\pi_{{\bf v}}$ at the origin is the same as that of the zero fibre of $g=u$ with
the model cuspidal edge $X$.

Suppose that the plane $\pi_{{\bf v}}$ is a member of the pencil of planes
that contains the tangential direction to $M$ at the origin (in particular, $\pi_{{\bf v}}$ is a tangent plane of the curve $\Sigma$).  If
$\pi_{{\bf v}}$ is not the tangent cone to $M$ at the origin, then the contact
of a generic $M$ with $\pi_{{\bf v}}$ is the same as that of the zero fibre of $g=\pm v\pm u^k$, $k=2,3,4$ with the model cuspidal edge $X$.
The integer $k$ is determined by the contact of $\pi_{{\bf v}}$ with the singular set $\Sigma$ (see \S \ref{sec:Geomsubmersions}).
The restriction of $H_{{\bf v}}$ to $\Sigma$ is given by
$$H_{{\bf v}}(x,0)=
\frac{1}{2}(a_{20}v_2+b_{20}v_3)x^2+\frac{1}{6}(a_{30}v_2+b_{30}v_3)x^3+
\frac{1}{24}(a_{40}v_2+b_{40}v_3)x^4+O(5).
$$
Therefore, the plane  $\pi_{{\bf v}}$ has an $A_k$, $k=1,2,3,$ contact with $\Sigma$ if and only if
$$
\begin{array}{rl}
A_1:&v_2a_{20}+v_3b_{20}\ne 0;\\
A_2:&v_2a_{20}+v_3b_{20}=0, \, a_{20}b_{30}-a_{30}b_{20}\ne 0;\\
A_3:&v_2a_{20}+v_3b_{20}=0, \,a_{20}b_{30}-a_{30}b_{20}= 0, \, a_{40}b_{20}-a_{20}b_{40}\ne 0.
\end{array}
$$
Geometrically, this mean that the plane  $\pi_{{\bf v}}$ has an $A_k$, $k=1,2,3,$ contact with $\Sigma$ if and only if
$$
\begin{array}{rl}
A_1:&\pi_{{\bf v}} \mbox{ \rm is not the osculating plane of }\Sigma;\\
A_2:&\pi_{{\bf v}} \mbox{ \rm is the osculating plane of }\Sigma,\, \tau_{\Sigma}(0)\ne 0;\\
A_3:&\pi_{{\bf v}} \mbox{ \rm is the osculating plane of }\Sigma,\, \tau_{\Sigma}(0)=0,\, \tau_{\Sigma}'(0)\ne 0.
\end{array}
$$

If the plane $\pi_{{\bf v}}$ coincides with the tangent cone to $M$
at the origin (i.e., ${\bf v}=(0,0,1)$)
but is not the osculating plane of $\Sigma$ (i.e., $\kappa_n(0)=b_{20}\ne 0$),
then the contact
of $M$ with $\pi_{{\bf v}}$ is the same as that of the zero fibre of $g=w\pm u^2$ with the model cuspidal edge $X$, that is, the height function has an $A_3$-singularity.

When $\pi_{{\bf v}}$ is the tangent cone to $M$ and coincides with the osculating
plane of $\Sigma$ ($\kappa_n(0)=b_{20}=0$) but
$\tau_{\Sigma}(0)\ne 0$ (i.e., $\kappa_i(0)=b_{30}\ne 0$) its contact with $M$ is described by the germ
$g=w+uv+au^3$ with the model cuspidal edge $X$. Here, the corresponding height function has a $D_4$-singularity.
(Compare with Theorem 2.11 in \cite{martinsnuno}
and Lemma 4.2 in \cite{teramoto}.) We observe that the case when the tangent cone to $M$
and the osculating plane of $\Sigma$ coincides at a point where $\tau_{\Sigma}(0)=0$ is not generic.

We have the sets
$$
\mathscr D_1(H)=\{({{\bf v}},H_{{\bf v}}(x,y))\in  S^2 \times \mathbb R:
\frac{\partial H_{{\bf v}}}{\partial x}=\frac{\partial H_{{\bf v}}}{\partial y}=0 \mbox{ at } (x,y,{{\bf v}})
\}
$$
and
$$
\mathscr D_2(H)=\{({{\bf v}},H_{{\bf v}}(x,0))\in S^2\times \mathbb R: \frac{\partial H_{{\bf v}}}{\partial x}=0 \mbox{ at  } (x,0,{{\bf v}})
\}.
$$

If $\pi_{\bf v}$ is a member of the pencil containing the tangential direction of $M$ but is not the tangent cone to $M$, then
the set $\mathscr D_1(H)$ coincides with $\mathscr D_2(H)$ and describes locally the dual of the curve $\Sigma$.
When $\pi_{\bf v}$ is the tangent cone to $M$, then
the set $\mathscr D_1(H)$ consists of two components. One of them is $\mathscr D_2(H)$ (the dual  of $\Sigma$) and the other is the {\it proper dual} of $M$ which
is the surface consisting of the tangent planes to $M$ away from points on $\Sigma$ together with their limits at points on $\Sigma$, i.e.,
the tangent cones at points on $\Sigma$.

If the contact of $M$ with  $\pi_{{\bf v}}$  is described by that of the fibre
$g=0$ with the model cuspidal edge $X$, with $g$ as in Theorem \ref{theo:Classification}, then $\mathscr D_1(H)$ (resp. $\mathscr D_2(H)$) is
diffeomorphic to $\mathscr D_1(F)$ (resp. $\mathscr D_2(F)$), where $F$ is an $\mathcal R^+(X)$-versal deformation of $g$ with 2-parametres.
In particular, the calculations and figures in \S\ref{sec:Geomsubmersions} give the models, up to diffeomorphisms, of $\mathscr D_1(H)$ and $\mathscr D_2(H)$.
We have thus the following result.

\begin{prop} Let $M$ be a generic cuspidal edge in $\mathbb R^3$. Then any height function on $M$ has locally
 one of the singularities modeled by the submersions in  {\rm Table \ref{tab:germsubm}}.
The proper dual of $M$ together with the dual of its singular curve $\Sigma$ are as in
{\rm Figure \ref{fig:Discwu2} (left)}  when the tangent cone to $M$ is distinct from the osculating plane to $\Sigma$ and as
{\rm Figure \ref{fig:Discwu2} (right)} otherwise.
\end{prop}

\section{Orthogonal projections of a cuspidal edge}\label{sec:projections}

The family of orthogonal projections in $\mathbb R^3$ is given by
$$
\begin{array}{cccc}
\Pi:&\mathbb R^3\times S^2&\to&TS^2\\
  &(p,{{\bf v}})          &\mapsto& (v, \Pi_{{\bf v}}(p))
\end{array}
$$
where $\Pi_{{\bf v}}(p)=p-(p \cdot {{\bf v}}){{\bf v}}$.
Given a surface $M$, we denote by $P$ the restriction
of $\Pi$ to $M$. Thus, for $M$ parametrised by $\phi$,
the family of orthogonal projections $ P:U\times S^2\to TS^2$ is given by
$$
P((x,y),{{\bf v}})=(v, P_{{\bf v}}(x,y)),
$$
with $P_{{\bf v}}(x,y)=\Pi_{{\bf v}}(\phi(x,y))=
\phi(x,y)-(\phi(x,y) \cdot {{\bf v}}){{\bf v}}$. The map $P_{{\bf v}}$ is
locally a map-germ from the plane to the plane and measures the contact of $M$ with lines parallel to ${\bf v}$.

We take $M$ parametrised as in \eqref{eq:prenormalform}. Consider,
for example, the projection along the tangential direction ${\bf
v}=(0,1,0)$. We have $P_{\mbox{\bf
v}}(x,y)=(x,b_1(x)+y^2b_2(x)+y^3b_3(x,y))$. If $\kappa_t=b_{12}=0$
and $\kappa_t'=b_{22}\neq 0$ the singularity of $P_{{\bf v}}$
at the origin is $\mathcal A$-equivalent to $(x,y^3-x^4y)$ and has
$\mathcal A_e$-codimension $3$ (see \S \ref{sec:Geometryprojections}
for even higher $\mathcal A_e$-codimension cases.) Thus, it cannot
be $\mathcal A_e$-versally unfolded by the family $P$. For this
reason, the group $\mathcal A$ is not very useful for
describing the singularities of the projections of $M$ and the way
they bifurcate as the direction of projection changes in $S^2$. We
follow here the same approach as that for the contact of the
cuspidal edge with planes. The projections are germs of submersions,
so we fix the $\mathcal A$-model $X$ of the cuspidal edge and
consider the action of the group $_{ X}\mathcal A=\mathcal
R(X)\times \mathcal L$ on $\mathcal M_3.\mathcal E(3,2)$. We classify germs of
submersions $g$ in $\mathcal M_3.\mathcal E(3,2)$ of $_{X}\mathcal A_e$-codimension $d_e(f, {_{X}\mathcal A})=\dim_{\mathbb R}(\mathcal E(3,2)/L_X{\mathcal A}_e{\cdot}f) \le 2$. We need the following results from \cite{bkd} and \cite{bdw} adapted to our group, where
$$
\begin{array}{rcl}
 L_X{\cal A}_1{\cdot}{g}&=& L\mathcal R_1(X)\cdot{} g+g^*({\mathcal M}_3).\{e_1,e_2\},\\
  L_X{\cal K}{\cdot}{g}&=& L\mathcal R(X)\cdot{} g+g^*({\mathcal M}_3).\{e_1,e_2\}.
 \end{array}
 $$

\begin{theo}\label{prop:CompTrans_XA}{\rm (\cite{bkd})}
Let $g:\mathbb R^3,0\to \mathbb R^2,0$ be a smooth germ and
$h_1,\ldots, h_r$ be homogeneous maps of degree $k+1$ with
the property that
$$
\mathcal{M}_3^{k+1}.\mathcal E(3,2) \subset L_X\mathcal A_1{\cdot}g+ sp\{
h_1,\ldots, h_r\} +\mathcal{M}_3^{k+2}.\mathcal E(3,2) .
$$
Then any germ $h$ with $j^kh(0)=j^kf(0)$ is $_X\mathcal A_1$-equivalent to a germ of the form
$g(x)+\sum_{i=1}^lu_ih_i(x)+\phi(x) $, where $\phi(x)\in \mathcal M_3^{k+2}.\mathcal E(3,2)$. The
vector subspace $sp\{ h_1,\ldots, h_r\}$ is called a complete
$(k+1)$-$_X\mathcal A$-transversal of $g$.
\end{theo}

\begin{theo}\label{theo:Detgroup_XA}
{\rm (\cite{bdw})}
If $g$ satisfies
$$\begin{array}{rcl}
\mathcal{M}_3^{l}.\mathcal{E}(3,2)&\subset&L_X\mathcal{K}\cdot{}g\\
\mathcal{M}_3^{k+1}.\mathcal{E}(3,2)&\subset&L_X\mathcal{A}_1\cdot{}g+\mathcal{M}_3^{l+k+1}.\mathcal{E
}(3,2)
\end{array}
$$
then $g$ is $k$-$_X\mathcal{A}$-determined.
\end{theo}

We also use the following lemma.

\begin{lem}{\rm \bf (Mather's Lemma)}
Let $\alpha:G\times M\to M$ be a smooth  action of a Lie group $G$ on a manifold $M$, and let $V$ be a connected
submanifold of $M$. Then $V$ is contained in a single orbit if and only if the following hold:

{\rm (a)} $T_vV\subseteq T_v(G.v), \forall v\in V$,

{\rm (b)} $\dim T_v(G_v)$ is independent of $v\in V$.
\end{lem}

We can now state the classification results for submersions in $\mathcal M_3.\mathcal E(3,2)$ of $_{ X}\mathcal A_e$-codimension $\le 2$ (of the stratum).

\begin{theo} \label{theo:ClassProj}
Let $X$ be the germ of the $\mathcal A$-model of the cuspidal edge parametrised by
$f(x,y,z)=(x,y^2,y^3)$. Denote by $(u,v,w)$ the coordinates in the target.
Then any germ of a submersion in $\mathcal{M}_3.\mathcal E(3,2)$ of $_{ X}\mathcal A_e$-codimension $\le 2$
is
$_{ X}\mathcal A$-equivalent to one of the  germs in {\rm Table \ref{tab:germProj}}.

\begin{table}[ht]
\caption{$_X\mathcal A$-finitely determined germs of $_X\mathcal A_e$-codimension $\le 2$.}
\begin{center}
{
\footnotesize
\begin{tabular}{llcl} \hline  Name & Normal form & $d_e(f,_{ X}\mathcal A)$ &$_{ X}\mathcal A_e$-versal deformation
\\
\hline
Type 1 & $(u,v)$ & $0$ &$(u,v)$\\
Type 2 &$(u,w+uv)$ & $0$ &$(u,w+uv)$\\
Type 3 &\textbf{}$(u,w+u^2v)$ & $1$ &$(u,w+u^2v+a_1v)$\\
Type 4 &$(u,w+u^3v)$ & $2$ &$(u,w+u^3v+a_2uv+a_1v)$\\
Type 5 &$(v+u^3,w+u^2)$ & $1$ &$(v+u^3+a_1u,w+u^2)$\\
Type 6 &$(v+u^5,w+u^2)$ & $2$ &$(v+u^5+a_2u^3+a_1u,w+u^2)$\\
Type 7 &$g_7=(v+au^2\pm u^4,w+uv+bu^3+P),$ & $2^{(*)}$ & $g_7+(0,a_1u+a_2u^2)$\\
& $P=cu^4+du^5+eu^6$&&\\
\hline
\end{tabular}
}
\end{center}
$(*)$: the codimension is of the stratum; $a,b,c,d,e$ are
moduli and are in the complement of some algebraic subsets of $\mathbb R^5$.
\label{tab:germProj}
\end{table}
\end{theo}

\begin{proof}
We follow the complete transversal technique and classify germs of
submersions inductively of the jet level. Consider the 1-jet
$g=(a_1u+b_1v+c_1w,a_2u+b_2v+c_2w)$. If $a_1\ne 0$ or $a_2\ne 0$
then $g$ is equivalent to $(u,av+bw)$, with $a\ne 0$ or $b\ne 0$
($g$ is a germ of a submersion). For $g=(u,av+bw)$,
$\xi_1(g)=(1,0)$, $\xi_2(g)=(0,2av+3bw)$, $j^1\xi_3(g)=(0,2aw)$, so
if $a\ne 0$, by applying Mather's Lemma, we get
 $g$ equivalent to $(u,v)$, otherwise it is equivalent to $(u,w)$.
 If $a_1=a_2=0$, then changes of coordinates in the target give $g$ equivalent to $(v,w)$. Thus, we have three orbits of submersions
 in the 1-jet space represented by
 $(u,v),\, (u,w),\, (v,w).$ The germ  $(u,v)$ is 1-determined and is stable.

$\bullet$ For $g=(u,w)$, we have $\xi_1(g)=(1,0)$, $\xi_2(g)=(0,w)$, $\xi_3(g)=(0,3v^2)$. Using these vectors and the left group, we can show that
 a  complete $(k+1)$-transversal is $(u,w+\lambda u^kv)$, $k\ge 1$.
 We have two orbits in the $(k+1)$-jet space, namely $(u,w+u^kv)$ and $(u,w)$. For $g=(u,w+u^kv)$,
 we have
 $
 \xi_1(g)=(1,ku^{k-1}v),\, \xi_2(g)=(0,3w+2u^kv),\, \xi_3(g)=(0,3v^2+2u^kw).
 $
 Using the above vectors and the left group, we can show that
$ \mathcal{M}_3^{2}.\mathcal{E}(3,2)\subset L_X\mathcal{K}\cdot{}g$. To prove that the germ is $k+1$-determined we need to prove that
\begin{equation}\label{eq:determinacy(u,w)}
\mathcal{M}_3^{k+2}.\mathcal{E}(3,2)\subset L_X\mathcal{A}_1\cdot{}g+\mathcal{M}_3^{k+4}.\mathcal{E}(3,2).
\end{equation}

 For this, we first show that all monomials of degree $k+3$ are in the right hand side of \eqref{eq:determinacy(u,w)}.
 Using $\xi_i$, $i=1,2,3$, we show that all the monomials of degree $k+3$ of the form $(P(u,v,w),0)$, $(0,wP(u,v,w))$ and $(0,v^2P(u,v,w))$ are in there. We use the left group to show that $(0,u^{k+3})$ is also in there.
If we write $g=(g_1,g_2)$, then $(0,g_1^ig_2)=(0,u^iw+u^{k+i}v)$ and $u^i\xi_2(g)=(0,3u^iw+2u^{k+i}v)$.
For $i\ge 1$ these vectors are in $L_X\mathcal{A}_1\cdot{}g$, so $(0,u^{k+2}v)$ is in the right hand side of \eqref{eq:determinacy(u,w)}. We proceed
similarly for the monomials of degree $k+2$ working now with vectors in $L_X\mathcal{A}_1\cdot{}g$ modulo elements in $\mathcal{M}_3^{k+3}.\mathcal{E}(3,2)$ to get all monomials of degree $k+2$ in the right hand side of \eqref{eq:determinacy(u,w)}. Therefore, \eqref{eq:determinacy(u,w)} holds.
The germ $g$ has $_X\mathcal A_e$-codimension $k-1$ and an $_X\mathcal A_e$-versal unfolding is
 $(u,w+u^kv+a_{k-1}u^{k}v+\ldots a_{1}uv)$.

 \medskip
 $\bullet$ For $g=(v,w)$, we have $\xi_1(g)=(0,0)$, $\xi_2(g)=(2v,3w)$, $\xi_3(g)=(2w,3v^2)$ and a complete 2-transversal is given by
 $g=(v+a_1uv+a_2u^2,w+b_1uv+b_2u^2)$. Then $\xi_1(g)=(a_1v+2a_2u,b_1v+2b_2u)$,
 $\xi_2(g)=(2v+2a_1uv,3w+2b_1uv)$, $\xi_3(g)=(2w+2a_1uw,3v^2+2b_1uw)$.
 Now $j^2w\xi_1(g)=(a_1vw+2a_2uw,b_1vw+2b_2uw)$ and $j^2u\xi_3(g)=(2uw,0)$ and we have $(vw,0)$ and $(0,vw)$ from the left group in the 2-jet of the $_X\mathcal A$ tangent space to the orbit of $g$, so if $b_2\ne 0$, we obtain $(0,uw)$.
 From this and $j^2u\xi_2(g)=(2uv,3uw)$, we also get $(uv,0)$. Then $j^2v\xi_1(g)=(a_1v^2+2a_2uv,b_1v^2+2b_2uv)$ together with $(v^2,0)$ and $(0,v^2)$ from the left group also gives $(0,uv)$ if $b_2\ne 0$. Therefore, by Mather's Lemma
 $g$ is equivalent to $(v+au^2,w+bu^2)$, with $b\ne 0$.
 Using the vectors $(w+bu^2,0)$, $\xi_3(g)=(2w,3v^2)$ and $(0,v^2)$ from the left group, shows that $g$ is equivalent to $(v,w+bu^2)$ if $b\ne 0$. We can then set $b=1$
  by changes of scales.
If $b_2=0$ and $b_1\ne 0$, then $g$ is equivalent to  $(v+au^2,w+uv)$, with $a$ a parameter modulus.
If $b_1=b_2=0$, the orbits are $(v\pm u^2,w)$, $(v+uv,w)$ and $(v,w)$ and all yield germs
of submersions of codimension (of the stratum) greater than 2. Thus, the 2 jets to consider are
$
(v,w+u^2)\,\mbox { \rm and }\, (v+au^2,w+uv).
$

Consider the germ $g=(v,w+u^2)$. Then  $\xi_1(g)=(0,2u)$, $\xi_2(g)=(2v,3w)$ and $\xi_3(g)=(2w,3v^2)$. Using these vectors and those from
the left group we can show that a complete $3$-transversal is $g=(v+\lambda u^3,w+u^2)$ and the orbits in the $3$-jet space are
$(v+u^3,w+u^2)$ and $(v,w+u^2)$.
The germ $(v+u^3,w+u^2)$ is $3$-determined and has codimension $1$. An $_X\mathcal A_e$-versal unfolding is given by
 $(v+u^3+a_1u,w+u^2)$.

The complete 4-transversal for  $(v,w+u^2)$ is empty and the orbits in the 5-jet are $(v+u^5,w+u^2)$ and $(v,w+u^2)$. The germ
$(v+u^5,w+u^2)$ is 5-determined and has codimension 2. An $_X\mathcal A_e$-versal unfolding is given by
 $(v+u^5+a_2u^3+a_1u,w+u^2)$.

For the 2-jet $(v+au^2,w+uv)$ a complete 3-transversal is given by
$g=(v+au^2+\lambda_1u^3,w+uv+\lambda_2u^3)$. Using Mather's Lemma,
it can be shown that g is equivalent to $g=(v+au^2,w+uv+bu^3)$,
where $b$ is also a modulus. A complete 4-transversal is
$g=(v+au^2\pm u^4,w+uv+bu^3+cu^4)$.
The computations here get too complicated to do by hand so we make use of the
computer package ``Transversal"  (\cite{transkirk}). It gives that
the complete 5-transversal is not empty and the orbits in the 6-jet space can be parametrised by
$g_7=(v+au^2\pm u^4,w+uv+bu^3+cu^4+du^5+eu^6)$. The same computer package shows
that the germ $g_7$ is 6-determined provided the moduli are not in the zero set of some polynomial.
The codimension of the stratum of $g_7$ is $2$ and an $_X\mathcal A_e$-versal unfolding is given by
 $g_7+(0,a_1u+a_2u^2)$.
\end{proof}


\subsection{Apparent contour of a cuspidal edge}\label{sec:Geometryprojections}

The singular set of an orthogonal projection of  a smooth surface in $\mathbb R^3$ along a direction $\bf v$
is the set of point where $\bf v$ is tangent to the surface and is called the {\it contour generator}.
The image of the contour generator by the projection
in the direction $\bf v$ is called the {\it apparent contour} ({\it profile}) of the surface along the directions $\bf v$.
 (See for example
\cite{gaffney, rieger} and also \cite{ShCarCidFaBook,koenderink} for more on apparent contours of smooth surfaces, \cite{brucegiblinProjBound} for surfaces with boundary,  \cite{taricreases} for those with creases and corners and
\cite{martinyutaro, west} for those of a surface with a cross-cap singularity.)

For a geometric cuspidal edge $M$, the projection is always singular along the singular curve $\Sigma$, so $\Sigma$ is always part of the contour generator and its
image is part of the apparent contour of $M$ along ${\bf v}$.  We call the {\it proper apparent contour $($profile$)$} of $M$
the projection of the set of points where ${\bf v}$ is tangent to $M$ at its regular points. We seek to describe the
apparent contour of $M$ and how it changes as the direction of projection changes locally in $S^2$.

\begin{theo} The bifurcations on the proper apparent contour of $M$ together with those of the
projection of the singular set of $M$ are, up to diffeomorphisms,  those in the following figures:

{\rm
\begin{tabular}{ll} \\
Type 2: &Figure \ref{fig:ProjType2}\\
Type 3: &Figure \ref{fig:ProjType3}\\
Type 4: &Figure \ref{fig:ProjType4}\\
Type 5: &Figure \ref{fig:ProjType5}\\
Type 6: &Figure \ref{fig:ProjType6}\\
Type 7: &Figures \ref{fig:ProjType7Lips(a)}, \ref{fig:ProjType7Lips(b)}, \ref{fig:ProjType7Lips(c)},  \ref{fig:ProjType7BeaksR7}, \ref{fig:ProjType7BeaksR6(2)} for some cases.\\
\end{tabular}
}
\label{theo:BiffAppCont}

For {\rm Type 1} singularities, the proper apparent contour is empty and the projection of the singular set is a regular curve.
\end{theo}

\begin{proof}
The apparent contour is the discriminant of the projection (that is, the image of its singular set). For a generic surface,
the family of projections is an $_{ X}\mathcal A_e$-versal family of the singularities of its members.
Therefore, the diffeomorphism type of the bifurcations of the apparent contour can be obtained by considering the bifurcations
 of the discriminants in the $_{ X}\mathcal A_e$-versal families restricted to $X$ in Theorem \ref{theo:Detgroup_XA}. We treat each case in Table \ref{tab:germProj} separately.

$\bullet$ Type 1: The germ $g=(u,v)$. We denote by $h$ the composite of $g$ with the parametrisation $f(x,y)=(x,y^2,y^3)$ of $X$. Then
$h(x,y)=(x,y^2)$ which has a singularity of $\mathcal A$-type fold.
The critical set of $h$ is the $x$-axis, i.e., is the singular set of $X$, and the discriminant is a regular curve.

\smallskip
$\bullet$ The germ $g=(u,w+u^kv)$, $k=1,2,3$. Here we have $h(x,y)=(x,y^3+x^ky^2)$ and its singular set is given by
$y(3y+2x^k)=0$. It has two components, one of which ($y=0$) is the singular set of $X$.
The other component (the proper contour generator) is a smooth curve and has $k$-point contact
with the singular set of $X$.

{Type 2: $k=1$.} The germ $h(x,y)=(x,y^3+xy^2)$ has a singularity of $\mathcal A$-type beaks
(which is of $\mathcal A_e$-codimension 1 but $g$ is $_X\mathcal A_e$-stable).
The discriminant is the union of the two curves $(x,0)$ and $(x,(4/27)x^3)$ which have 3-point contact at the origin \mbox{(Figure \ref{fig:ProjType2})}.

\begin{figure}[h]
\begin{center}
\includegraphics[width=3cm, height=3cm]{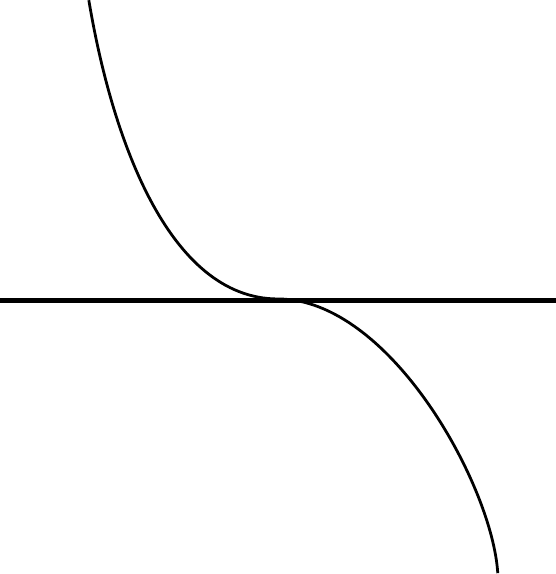}
\caption{Proper apparent contour (thin) and the projection of the singular set (thick) of $M$ at an $_X\mathcal A$-singularity of Type 2.}
\label{fig:ProjType2}
\end{center}
\end{figure}

{Type 3: $k=2$.} We
consider the versal deformation $g=(u,w+u^2v+a_1uv)$ with parameter $a_1$, so $h(x,y)=(x,y^3+x^2y^2+a_1y^2)$. (When $a_1=0$,
$h$ is $\mathcal A$-equivalent to $(x,y^3-x^4y)$ which is a singularity of
 $\mathcal A$-type $4_4$, see \cite{rieger}.)
For $a_1$ fixed, the critical set is given by $y(3y +2x^2+2a_1)=0$ and consists of two curves $y=0$ and $y=-2/3(x^2+a_1)$.
The discriminant if the union of the two curves $(x,0)$ and $(x,4/27(x^2+a_1)^3)$.
See Figure \ref{fig:ProjType3} for the bifurcations in these curves as $a_1$ varies near zero.

\begin{figure}[h]
\begin{center}
\includegraphics[width=9cm, height=3cm]{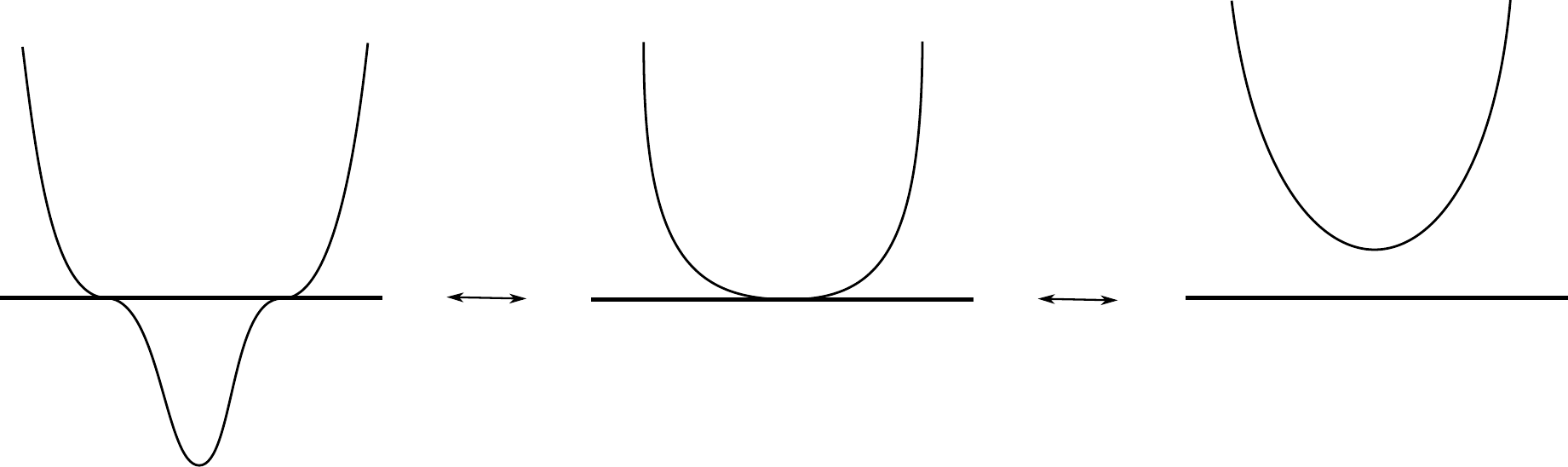}
\caption{Bifurcations in generic 1-parameter families of the proper apparent contour (thin) and of the projection of the singular set (thick) of $M$ at an $_X\mathcal A$-singularity of Type 3.}
\label{fig:ProjType3}
\end{center}
\end{figure}

{Type 4: $k=3$:}
We consider the versal deformation $g=(u,w+u^3v+a_2uv+a_1v)$ with parameters $a_1, a_2$,
so $h(x,y)=(x,y^3+x^3y^2+a_2xy^2+a_1y^2)$. (At $a_1=a_2=0$, $f$ is $\mathcal A$-equivalent to $(x,y^3-x^6y)$ which is a singularity of
 $\mathcal A$-type $4_6$.)
The critical set is given by $y(3y +2(x^3+a_2x+a_1))=0$. The family $x^3+a_2x+a_1$ is a $\mathcal K$-versal deformation of the $A_2$-singularity, so
the bifurcations in the critical set are as in \mbox{Figure \ref{fig:ProjType4} (left)} and those in the discriminant are as in Figure \ref{fig:ProjType4} (right). Singularities of Type 3 occur when $(a_1,a_2)$ are on the cusp curve $27a_1^2+4a_2^3=0$
(middle figures in Figure \ref{fig:ProjType4}).

\begin{figure}[h]
\begin{center}
\includegraphics[width=6cm, height=6cm]{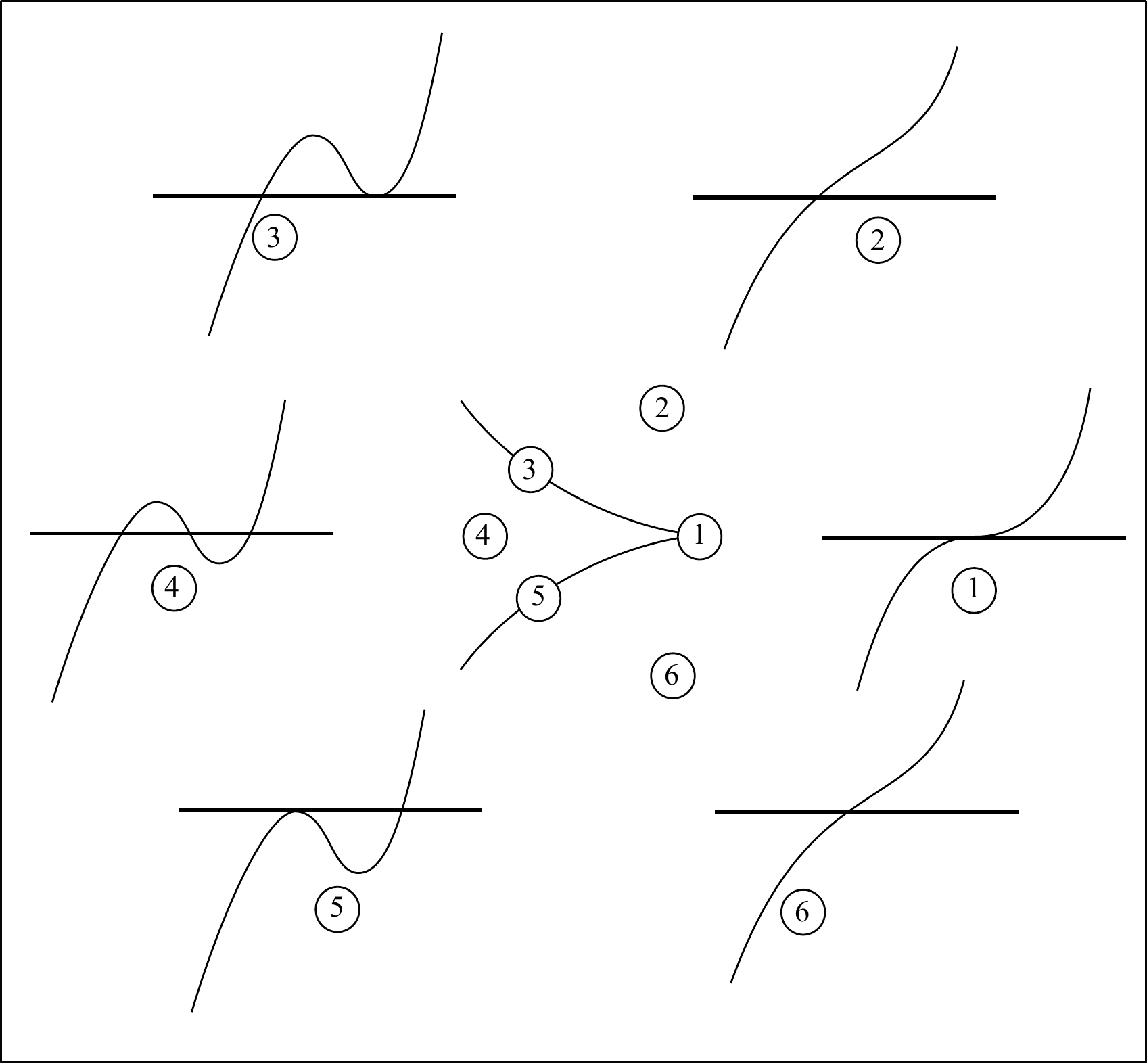}\qquad
\includegraphics[width=6cm, height=6cm]{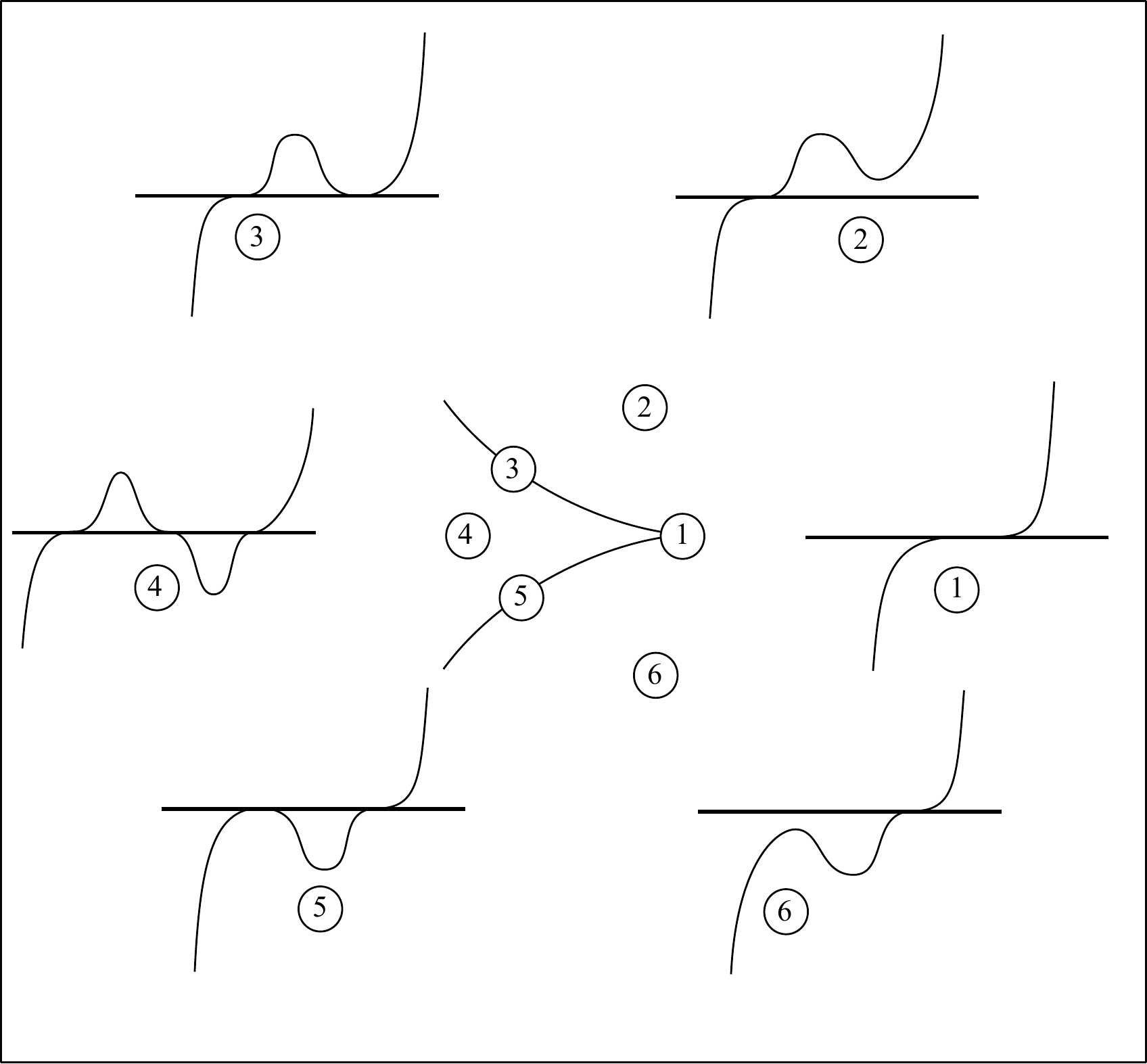}
\caption{Bifurcations in generic 2-parameter families of
the critical set (left), and of the proper apparent contour (thin) and of the projection of the singular set (thick) of $M$ (right)
at an $_X\mathcal A$-singularity of Type 4.}
\label{fig:ProjType4}
\end{center}
\end{figure}

\medskip
$\bullet$ Type 5: Here we have a versal family $g=(v+u^3+a_1u,w+u^2)$, so $h(x,y)=(y^2+x^3+a_1x,y^3+x^2)$. (When $a_1=a_2=0$,
$f$ has an $\mathcal A$-singularity of type $\mbox{\rm I}_{2,2}^{1,1}$, see \cite{riegerruas}.)  The critical set is given by
$y(4x-3y(3x^2+a_1))=0$ and is the union of two transverse curves.
The image of the curve $y=0$ is the $\mathcal A_e$-versal family
$(x^3+a_1x,x^2)$ of a cusp curve (Figure \ref{fig:ProjType5}, thick curve).
The image of the other branch is a cusp when $a_1=0$, and when $a_1\ne 0$, we write $y=4x/(3(3x^2+a_1))$ so its image can be parametrised by
$$
(\frac{16}{9}\frac{x^2}{(3x^2+a_1)^2}+x^3+a_1x, \frac{64}{27}\frac{x^3}{(3x^2+a_1)^3}+x^2).$$
A short calculation shows that it has always a cusp singularity near the origin for all values of $a_1\ne 0$ near zero.
The origin is a Type 2 singularity for any $a_1\ne 0$; see
Figure \ref{fig:ProjType5}.

\begin{figure}[h]
\begin{center}
\includegraphics[width=9cm, height=2cm]{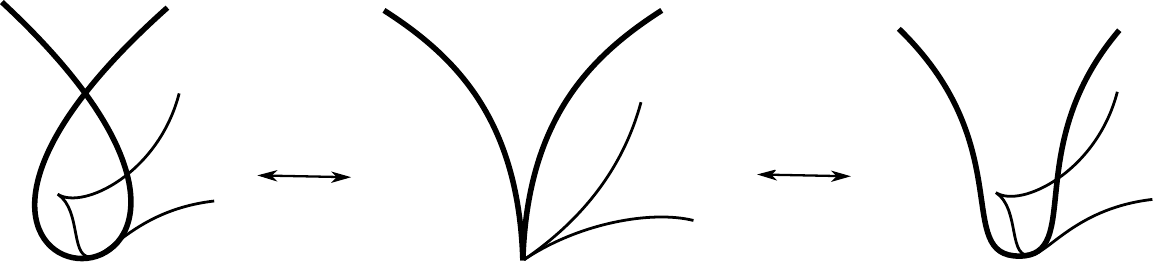}
\caption{Bifurcations in generic 1-parameter families of the proper apparent contour (thin) and
of the projection of the singular set (thick) of $M$ at an $_X\mathcal A$-singularity of Type 5.}
\label{fig:ProjType5}
\end{center}
\end{figure}

$\bullet$ Type 6: The versal family $g=(v+u^5+a_2u^3+a_1u,w+u^2)$ gives
$h(x,y)=(y^2+x^5+a_2x^3+a_1x,y^3+x^2)$ (which has a singularity of $\mathcal A$-type
$\mbox{\rm I}_{2,2}^{1,2}$ at $a_1=a_2=0$).
The critical set is given by $y(4x-3y(a_1+3a_2x^2+5x^4))=0$ and consists of two transverse curves
for any $(a_1,a_2)$ near the origin.
The image of $y=0$ is the $\mathcal A_e$-versal family $(x^5+a_2x^3+a_1x,x^2)$
of the ramphoid cusp curve. The $\mathcal A_e$-deformations in this family are obtained in
\cite{GibsonHobbs, wallgenericgeometry} and are as in Figure \ref{fig:ProjType6}, thick curves. One can show that there are no other
local or multi-local singularities appearing in the deformation.

\begin{figure}[h]
\begin{center}
\includegraphics[width=8cm, height=8cm]{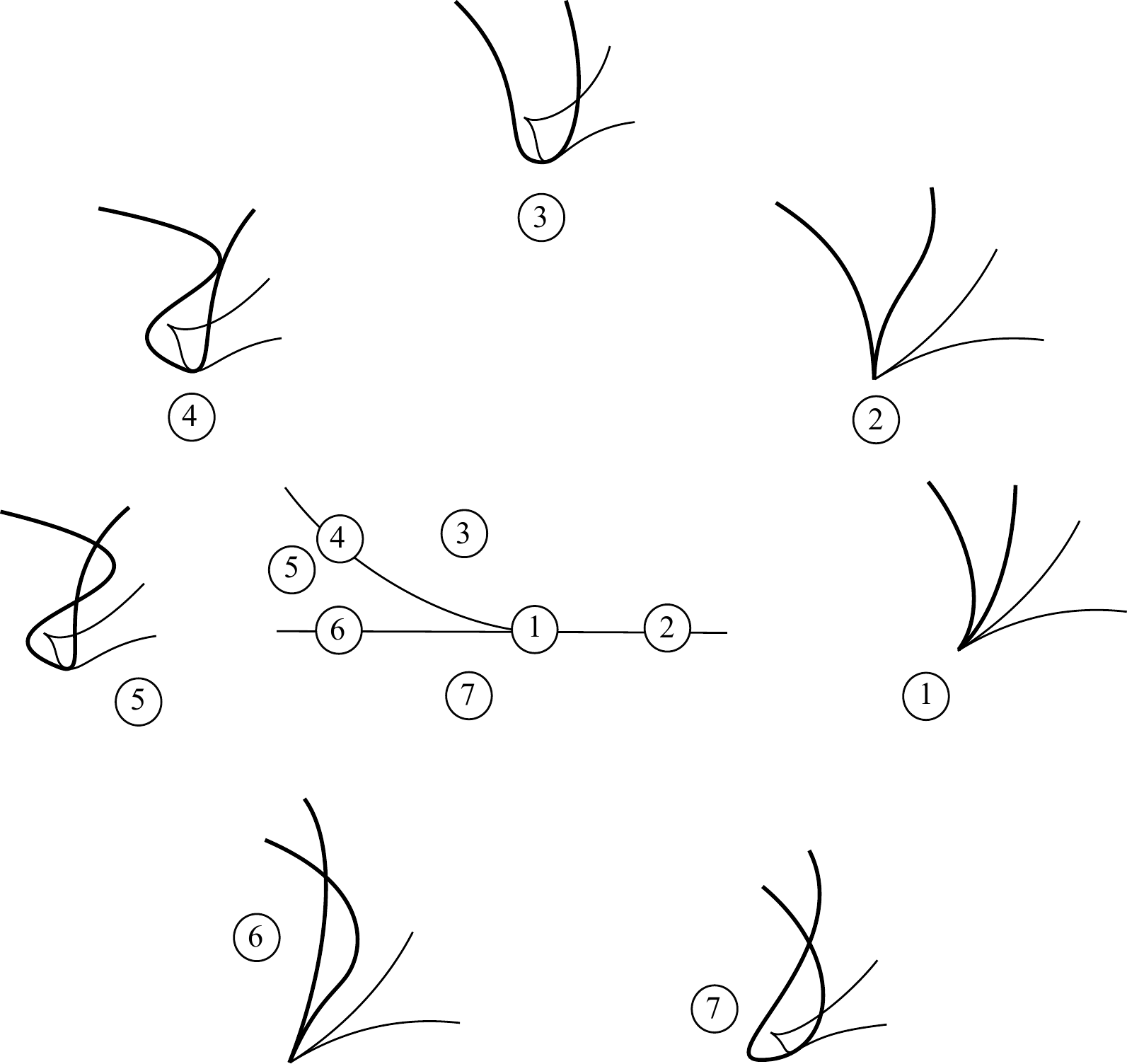}
\caption{Stratification of the parameter space (central figure) and bifurcations in generic 2-parameter families of the proper apparent contour (thin) and of the projection of the singular set (thick) of $M$ at an $_X\mathcal A$-singularity of Type 6.}
\label{fig:ProjType6}
\end{center}
\end{figure}

\medskip
$\bullet$ Type 7:
 The versal family $(v+au^2\pm u^4,w+uv+bu^3+cu^4+du^5+eu^6+a_2u^2+a_1u)$ gives
$h(x,y)=(y^2+ax^2\pm x^4,y^3+xy^2+bx^3+cx^4+dx^5+ex^6+a_2x^2+a_1x)$ (which has
a non-simple corank 2 singularity when $a_1=a_2=0$; these are yet to be classified).
We shall stratify the parameter space $(a_1,a_2)$ near the origin by the loci of codimension 2,1,0 local and multi-local
singularities of $h$. The stratification depends on the
 five moduli of $g_7$ in Table \ref{tab:germProj}. However, as shown by the calculations below,
the configuration of the strata (up to homeomorphism $\mathbb R^2,0\to \mathbb R^2,0$)
depends only on the two moduli $a$ and $b$. We also obtain a (partial) stratification of the $(a,b)$ plane
into strata where the configuration of the bifurcation set of $h$ is constant.

The singular set of $h$ is given by
$$
y\left(a_1+2a_2x-(2a-3b)x^2-3axy+y^2+4cx^3+(5d \mp4)dx^4 \mp 6yx^3+6ex^5\right)=0.
$$
It consists of the singular set of the cuspidal edge $y=0$ and another component which we denote by $S_{(a_1,a_2)}$.
The image of $y=0$ gives an $\mathcal A_e$-versal family
$$
(ax^2\pm x^4,bx^3+cx^4+dx^5+ex^6+a_2x^2+a_1x)
$$
of a cusp curve. (We require $a\ne 0$ and $b\ne 0$ which are also conditions for finite determinacy of the germ $g_7$.
We observe that we get a self-intersection in the image of $y=0$ if and only if $a_1b<0$.)

\begin{figure}[h]
\begin{center}
\includegraphics[width=7cm, height=6cm]{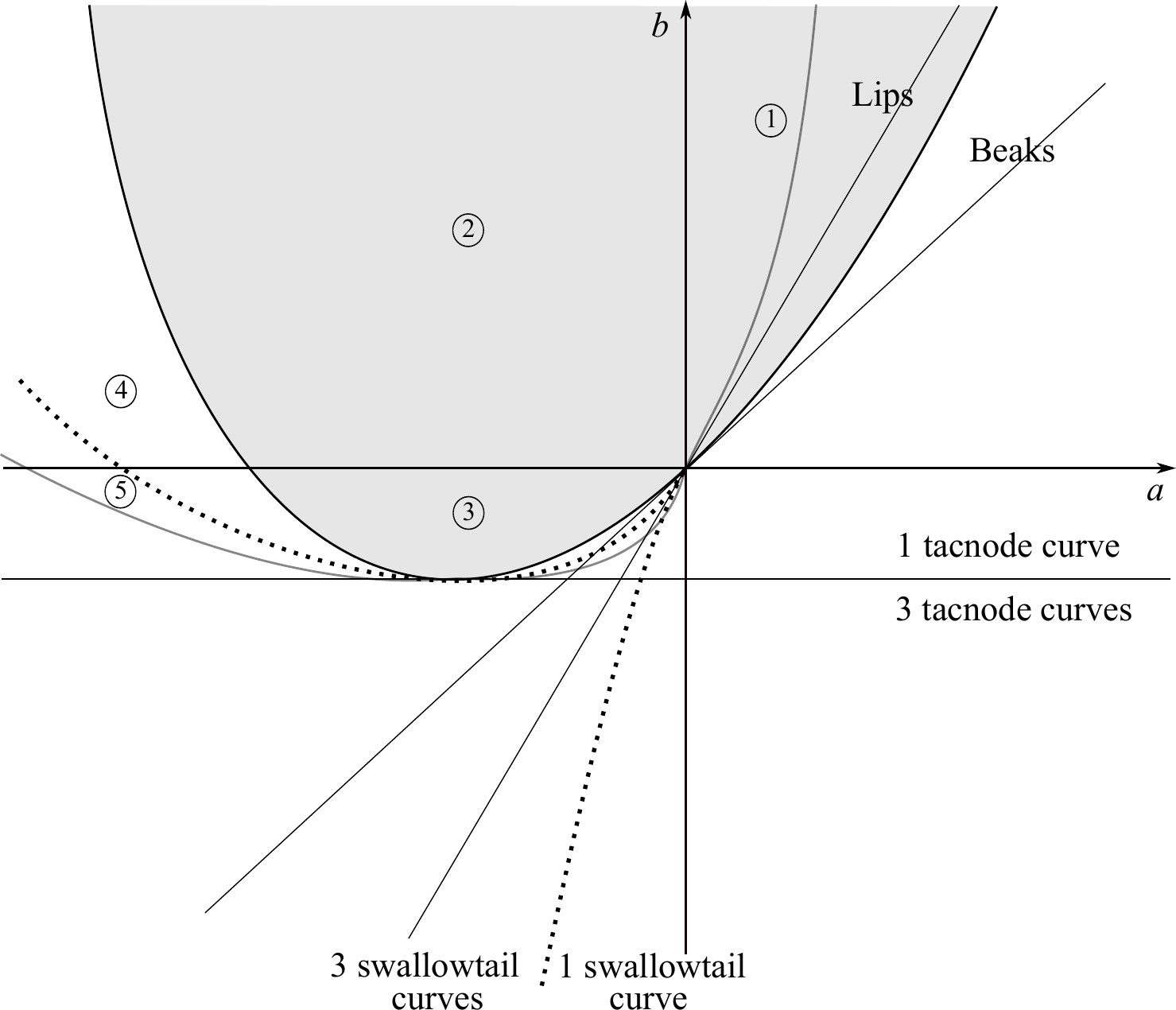}
\caption{A (partial) stratification of the $(a,b)$-plane.}
\label{fig:Type7Strat(ab)}
\end{center}
\end{figure}

\medskip
\noindent
{\it The lips/beaks stratum}\\
 The singular set $S_{(0,0)}$  has a Morse singularity at the origin if and only if
$9 a^2+8 a-12 b\ne 0.$ We obtain a stratum in the $(a,b)$-plane given by  the curve
\begin{equation}
\label{eq:MorseCond}
9 a^2+8 a-12 b=0,
\end{equation}
where the singularity of $S_{(0,0)}$ is more degenerate.
The parabola \eqref{eq:MorseCond} splits the $(a,b)$-plane into two regions. For $(a,b)$ inside (resp. outside) the parabola,
we have a birth of a lips (resp. beaks) singularity on one of the regular sheets of the cuspidal edge (see Figure \ref{fig:Type7Strat(ab)}).
We call the singularities of Type 7 lips (resp. beaks) type if lips (resp. beaks) singularities appear in its bifurcations.

The lips/beaks stratum in the $(a_1,a_2)$-parameter space is given by
the set of parameters $(a_1,a_2)$ for which there exists $(x,y)$ near the origin such that
$S_{(a_1,a_2)}(x,y)=0$ and the function $S_{(a_1,a_2)}$ has a Morse singularity at $(x,y)$. Eliminating variables, we find that
the lips/beaks stratum is a regular curve parametrised by
\begin{equation}
\label{eq:LipsBeaksStrat}
a_1=-\frac{4}{9 a^2+8 a-12 b}\, a_2^2+O_3,
\end{equation}
where $O_3$ (here and in the rest of the paper) is a family of smooth functions in $a_2$ depending smoothly on $a,b,c,d,e$ and has a zero 2-jet as a function in $a_2$.

\medskip
\noindent
{\it Swallowtail stratum} \\
 Using the recognition criteria in \cite{kentaro},
we find the stratum where swallowtail singularities occur is given by
\begin{equation}
\label{eq:SwallotStrat}
a_1=\frac{4(3a^2+(3b-4a)\alpha -\alpha^3)\alpha}{9(a^2+2(b-a) \alpha -\alpha^2 a)^2}\, a_2^2+O_3,
\end{equation}
where $\alpha$ is a solution of the cubic equation
$$
P(\alpha)=2 (a-b)\alpha^3-3 a^2 \alpha^2-a^3=0.
$$

The discriminant of $P$ is $\delta_P=a^3+(a-b)^2$. It is a cusp curve tangent to the lips/beaks parabola
at $(-\frac{4}{9},-\frac{4}{27})$ where both curves have a horizontal tangent (dotted curve in Figure \ref{fig:Type7Strat(ab)}). There are thus three swallowtail
curves in the bifurcation set of $h$ for $(a,b)$ inside the cusp curve $\delta_P=0$ and one swallowtail
curve for $(a,b)$ outside this curve (see Figure \ref{fig:Type7Strat(ab)}). In particular, when lips singularities occur on the profile,
only one swallowtail curve is present in the bifurcation set of $h$. From \eqref{eq:SwallotStrat} we have two additional
strata in the $(a,b)$ plane given as follows:

(i) the swallowtail curve is inflectional, so $\alpha$ is also a root of $3a^2+(3b-4a)\alpha -\alpha^3$.
Calculating the resultant of this polynomial and $P$ we get $243a^4+4(4a-3b)^3=0$ (gray curve in Figure \ref{fig:Type7Strat(ab)}).

(ii) the swallowtail curve is singular, so $\alpha$ is also a root of $a^2+2(b-a) \alpha -\alpha^2 a$.
Calculating the resultant of this polynomial and $P$ we get $a^3(a^3+(a-b)^2)^2=0$, which gives curves that are
already present in the stratification.

\medskip
\noindent
{\it Type 3 singularities stratum}\\
 These occur when $S_{(a_1,a_2)}=0$ is tangent to $y=0$. This occurs
when
$$
a_1=\frac{1}{3b-2a}\, a_2^2+O_3.
$$

Here we require $3b-2a\ne 0$, so we have another stratum in the $(a,b)$-plane given by $3b-2a=0$.
This is precisely the tangent line to the lips/beaks stratum at the origin.

\medskip
\noindent
{\it Type 5 singularities stratum}\\
 These occur when the image of the singular set $h(0,y)$
is singular, and this happens when $a_1=0$ (for any value of the moduli).

\medskip
The above strata exhaust all the possible codimension 1 local singularities that can occur in $h$. We turn now to the multi-local singularities.

\medskip
\noindent
{\it Douple point $+$ fold stratum}\\
 Here a proper profile passes through the
point of self-intersection (the double point) of the image of the singular set of the cuspidal edge.
Thus we have $h(x_1,0)=h(x_2,0)$ for some $x_1\ne x_2$, $S_{(a_1,a_2)}(x,y)=0$ and $h(x_1,0)=h(x,y)$.

From $h(x_1,0)=h(x_2,0)$ we get $x_2=-x_1$ and $a_1+bx_1^2+dx_1^4=0$, so $a_1=-bx_1^2-dx_1^4$. We have now a system of
five equations $S_{(-bx_1^2-dx_1^4,a_2)}(x,y)=0$ and $h(x_1,0)=h(x,y)$ in $x_1,x,y,a_2$. A necessary condition for the existence of a
solution is $a^3+(a-b)^2>0$ and $a\ne b$. Thus, there are no ``douple point $+$ fold" singularities for $(a,b)$
inside the cusp $\delta_P=a^3+(a-b)^2=0$ where 3 swallowtail curves appear. When  $a^3-(a-b)^2<0$,
calculations show that the  ``douple point $+$ fold" stratum is a regular curve parametrised
by
$$
a_1=-\frac{b}{a^3+(a-b)^2}\, a_2^2+O_3.
$$

\medskip
\noindent
{\it Double point $+$ Type 2 stratum}\\
We have a double point on the image of the singular set of the cuspidal edge
and a fold tangent to one of its branches at the double point. Thus
$h(x_1,0)=h(x_2,0)$ for some $x_2=-x_1$ (from above stratum) and $S_{(a_1,a_2)}(x_1,0)=0$ (or $S_{(a_1,a_2)}(-x_1,0)=0$).
The stratum is again a regular curve parametrised by
$$
a_1=-\frac{b}{(a-b)^2}\, a_2^2+O_3.
$$

\medskip
\noindent
{\it Type 2 $+$ fold stratum}\\
Here we a have a Type 2 singularity together with another piece of the proper profile intersecting the two
tangential curves transversally. Thus, we have a Type 2 singularity at $(x,0)$ and there exists $(x_1,y_1)$, with $y_1\ne 0$ such that
$S_{(a_1,a_2)}(x,0)=S_{(a_1,a_2)}(x_1,y_1)=0$ and $h(x,0)=h(x_1,y_1)$.

This stratum consists of 1 or 3 regular curves depending on whether the polynomial
$$
Q(\lambda)=(a^3+(a-b)^2)\lambda^3-3(a^3+(a-b)^2)\lambda^2+3(a-b)^2\lambda+4a^3-(a-b)^2
$$
has 1 or 3 roots.  The discriminant of $Q$ is $-108a^6(a-b)^2(a^3+(a-b)^2)$. In particular, we have 1 (resp. 3) regular curves
if and only if the swallowtail stratum consists of 1 (resp. 3) regular curve(s).

The initial part of the parametrisation of the curve(s) is
$$
a_1=\frac{4(\lambda-1)((3b-3a)\lambda^2-(3a\mu +3b -2a)\lambda+a)}{((3b-3a)\lambda^2-3a\mu\lambda-3b+3a)^2}\, a_2^2+O_3,
$$
with $\mu=-\sqrt{a(1-\lambda^2)}$ and $\lambda$ a root of the  polynomial $Q$.

\medskip
\noindent
{\it Triple point stratum} \\
We have three pieces of the proper apparent contour intersecting transversally at a given point,
so we need to solve $h(x_1,y_1)=h(x_2,y_2)=h(x_3,y_3)$ in  $(a_1,a_2)$, with $(x_i,y_i)\ne (x_j,y_j)$ for $i\ne j$.
The stratum is empty in the examples in Figures \ref{fig:ProjType7Lips(a)}, \ref{fig:ProjType7Lips(b)}, \ref{fig:ProjType7Lips(c)},  \ref{fig:ProjType7BeaksR7}  and consists of a regular curve in the case of Figure \ref{fig:ProjType7BeaksR6(2)}. For the general case, the equations are too lengthy
to reproduce here and eliminating variables leads to equations of very high order.

\medskip
\noindent
{\it Tacnode stratum}\\
 This consists of a multi-local singularity where the proper apparent contour has an ordinary tangency with
the image of the cuspidal curve.
The stratum consists of one regular curve if $b>-\frac{4}{27}$ and 3 regular curves if $b<-\frac{4}{27}$ (and $ab\ne 0$) parametrised by
$$
a_1=\frac{(-a\lambda^4+2b\lambda^3+(a\mu+2a-3b)\lambda^2-a\mu-a+b)(\lambda-1)^2}
{((b-a)\lambda^3-a\mu\lambda^2+(3a\mu+3a-3b)\lambda-2a\mu+2b-2a)^2}\, a_2^2+O_3,
$$
with $\lambda$ a solution of
$$(a^3-(a+b)^2)\lambda^3+(3a^3+4b^2-(a-b)^2)\lambda^2+(3a^3-3b^2+(a+b)^2)\lambda+a^3+(a-b)^2=0$$
and $\mu=-\sqrt{-a(1-\lambda^2)}$.

\medskip
\noindent
{\it Cusp $+$ fold}\\
We have a cusp on the proper apparent contour on the image of the singular set of the cuspidal edge. Here we get two
regular curves given by
$$
a_1=\frac{4}{9}\frac{(3a^2+(3b-4a)\beta-\beta^3)\beta}{(a^2+2(b-a)\beta-a\beta^2)^2}\,a_2^2+O_3,
$$
with $\beta$ one of the two real roots of
$$
3a^5-4a^3(a-b)\beta+6\beta^2a^4-12a^2(a-b)\beta^3+(4(a-b)^2-a^3)\beta^4.
$$

We draw in Figures \ref{fig:ProjType7Lips(a)}, \ref{fig:ProjType7Lips(b)}, \ref{fig:ProjType7Lips(c)}
three possibilities for the bifurcations of the proper apparent contour and of the projections of the singular set
for the lips Type 7 singularity. Figures \ref{fig:ProjType7BeaksR7}, \ref{fig:ProjType7BeaksR6(2)}
are for the bifurcations of a  beaks Type 7 singularity with one figure having one swallowtail stratum and the other 3 swallowtail strata.
As Figure \ref{fig:Type7Strat(ab)} shows, there are various open strata in the $(a,b)$-plane and in each stratum we have distinct
bifurcations of the apparent contour. We observe that the stratification in \mbox{Figure \ref{fig:Type7Strat(ab)}} is a partial one as we have not included, for instance,
where the various curves in the stratification of the $(a_1,a_2)$-plane are inflectional nor where their relative position changes.
Figures \ref{fig:ProjType7Lips(a)}-\ref{fig:ProjType7BeaksR6(2)} and \mbox{Figure \ref{fig:Type7Strat(ab)}} show the richness of  the extrinsic differential geometry of the cuspidal edge.
\end{proof}

\begin{rem}{\rm
In the calculations in the proof of Theorem \ref{theo:BiffAppCont}  for the Type 7 singularity, we eliminate variables using resultant (using Maple)
until we get
get an equation $k=0$ involving only two variables. The function germ $k$ is finitely $\mathcal R$-determined and has a well understood singularity (in the cases involved), so we can deduce the smooth structure of $k=0$ and, in particular, the number of its branches and the initial parametrisation of each branch.
We reverse the process until we get the initial part of the local parametrisation of the desired stratum.
}
\end{rem}

\begin{figure}[h]
\begin{center}
\includegraphics[width=15cm, height=20cm]{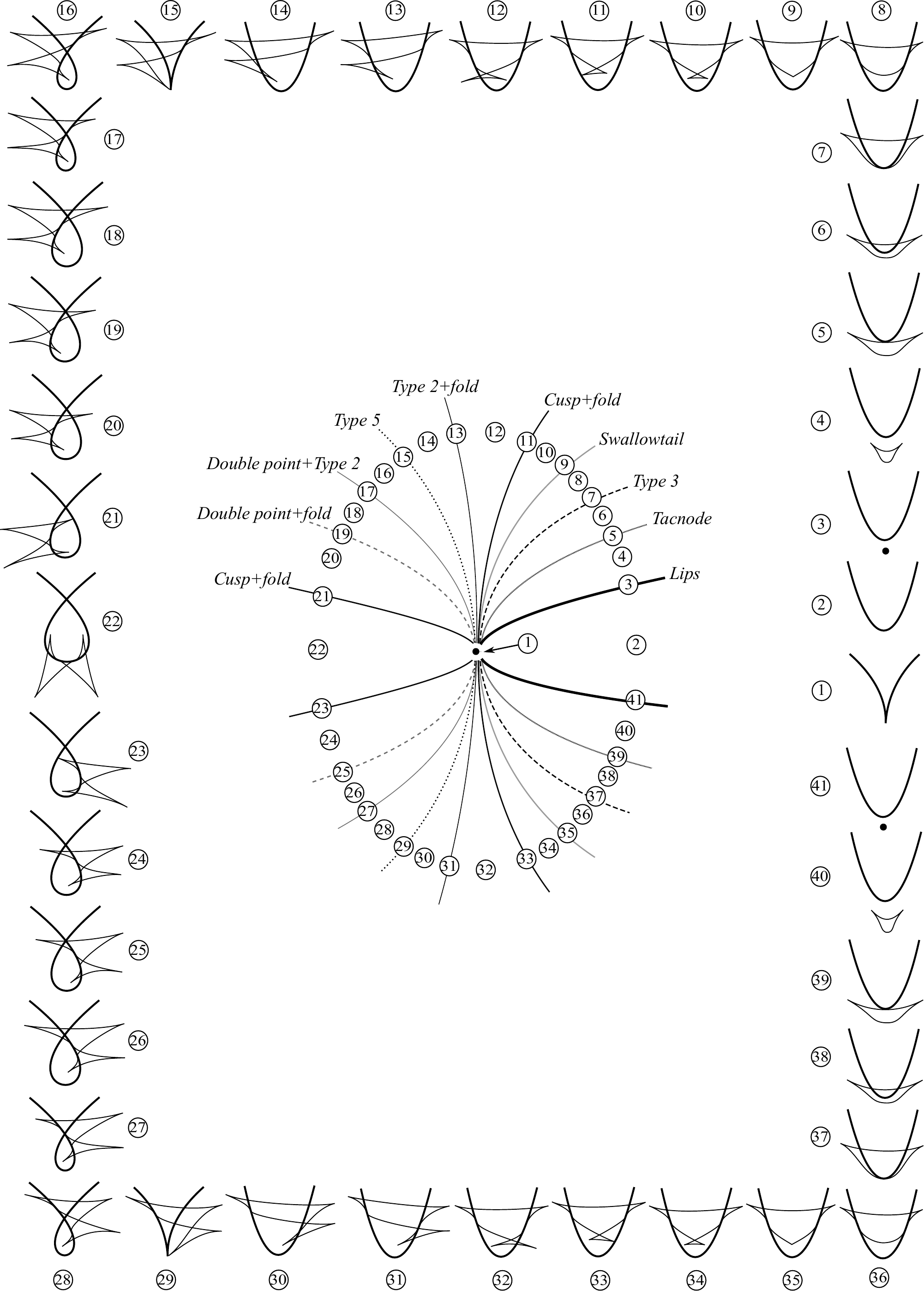}
\caption{Lips type bifurcations in the proper apparent contour and of the projections of the singular set at a Type 7 singularity with
$(a,b)$ in Region \textcircled{$\scriptstyle1$}, in Figure \ref{fig:Type7Strat(ab)}.}
\label{fig:ProjType7Lips(a)}
\end{center}
\end{figure}

\begin{figure}[h]
\begin{center}
\includegraphics[width=15cm, height=20cm]{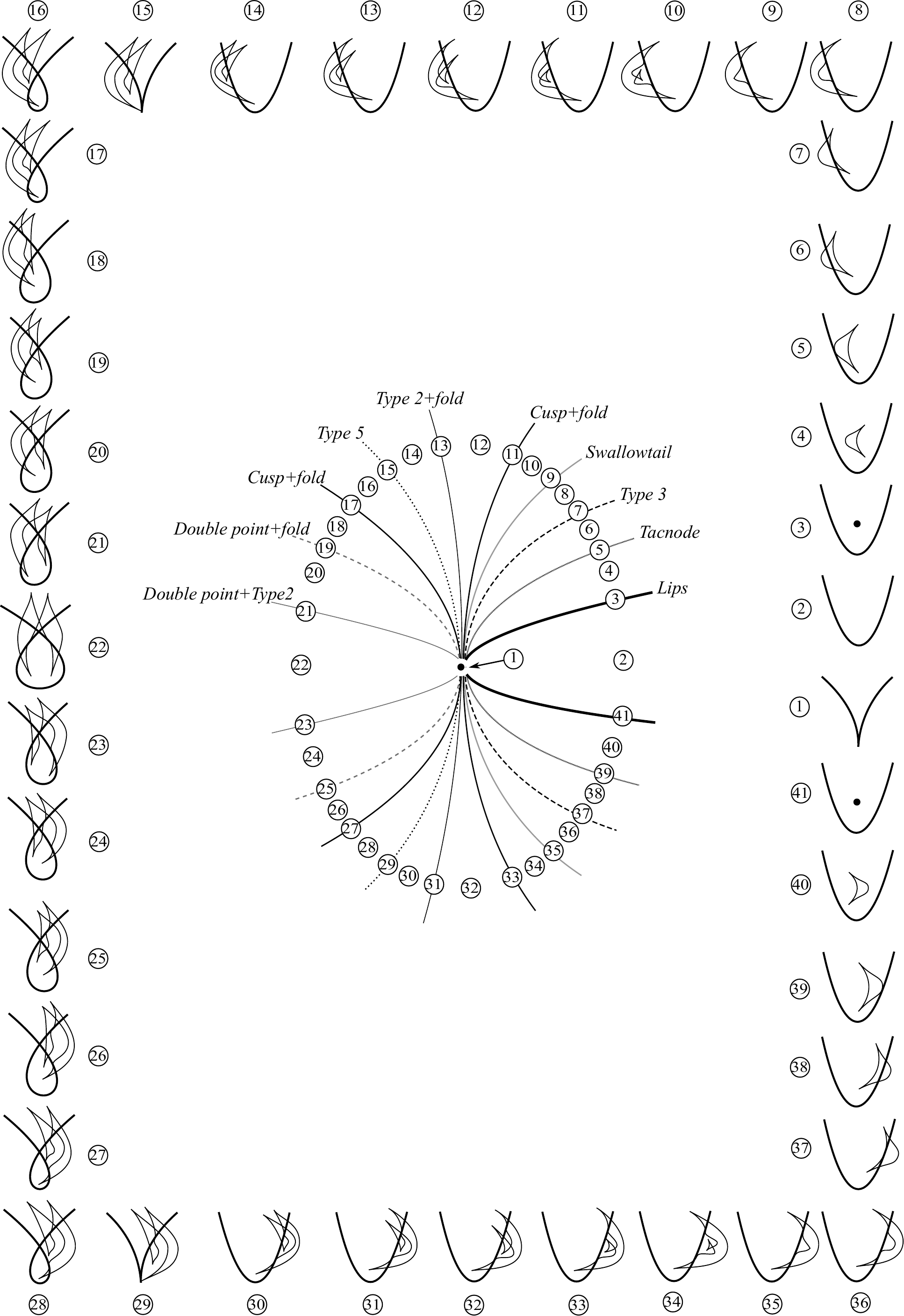}
\caption{Lips type bifurcations in the proper contour and of the projections of the singular set at a Type 7 singularity with
$(a,b)$ in Region \textcircled{$\scriptstyle2$}, in Figure \ref{fig:Type7Strat(ab)}.}
\label{fig:ProjType7Lips(b)}
\end{center}
\end{figure}

\begin{figure}[h]
\begin{center}
\includegraphics[width=15cm, height=20cm]{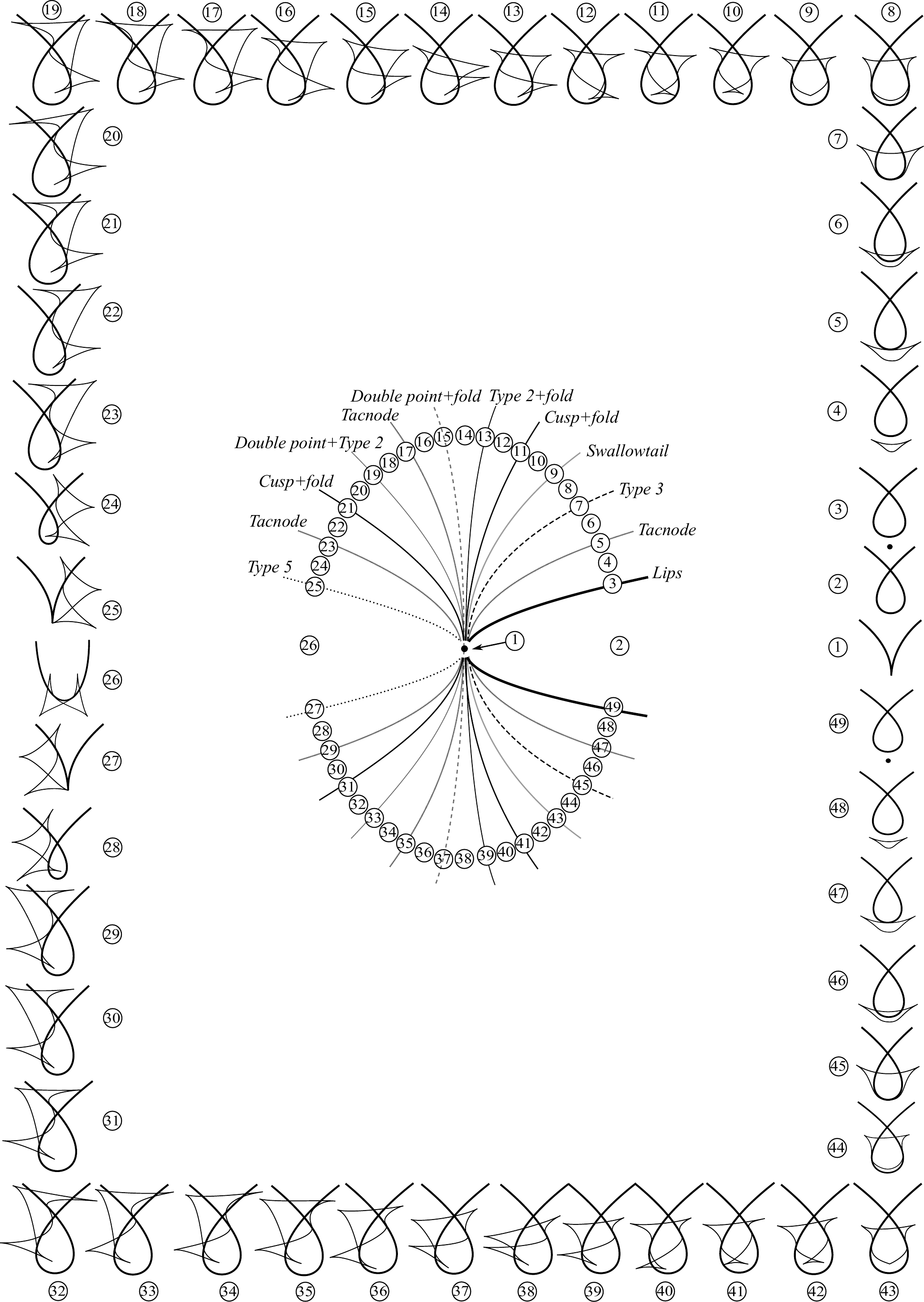}
\caption{Lips type bifurcations of the proper contour and of the projections of the singular set at a Type 7 singularity with
$(a,b)$ in Region \textcircled{$\scriptstyle3$}, in Figure \ref{fig:Type7Strat(ab)}.}
\label{fig:ProjType7Lips(c)}
\end{center}
\end{figure}

\begin{figure}[h]
\begin{center}
\includegraphics[width=15cm, height=20cm]{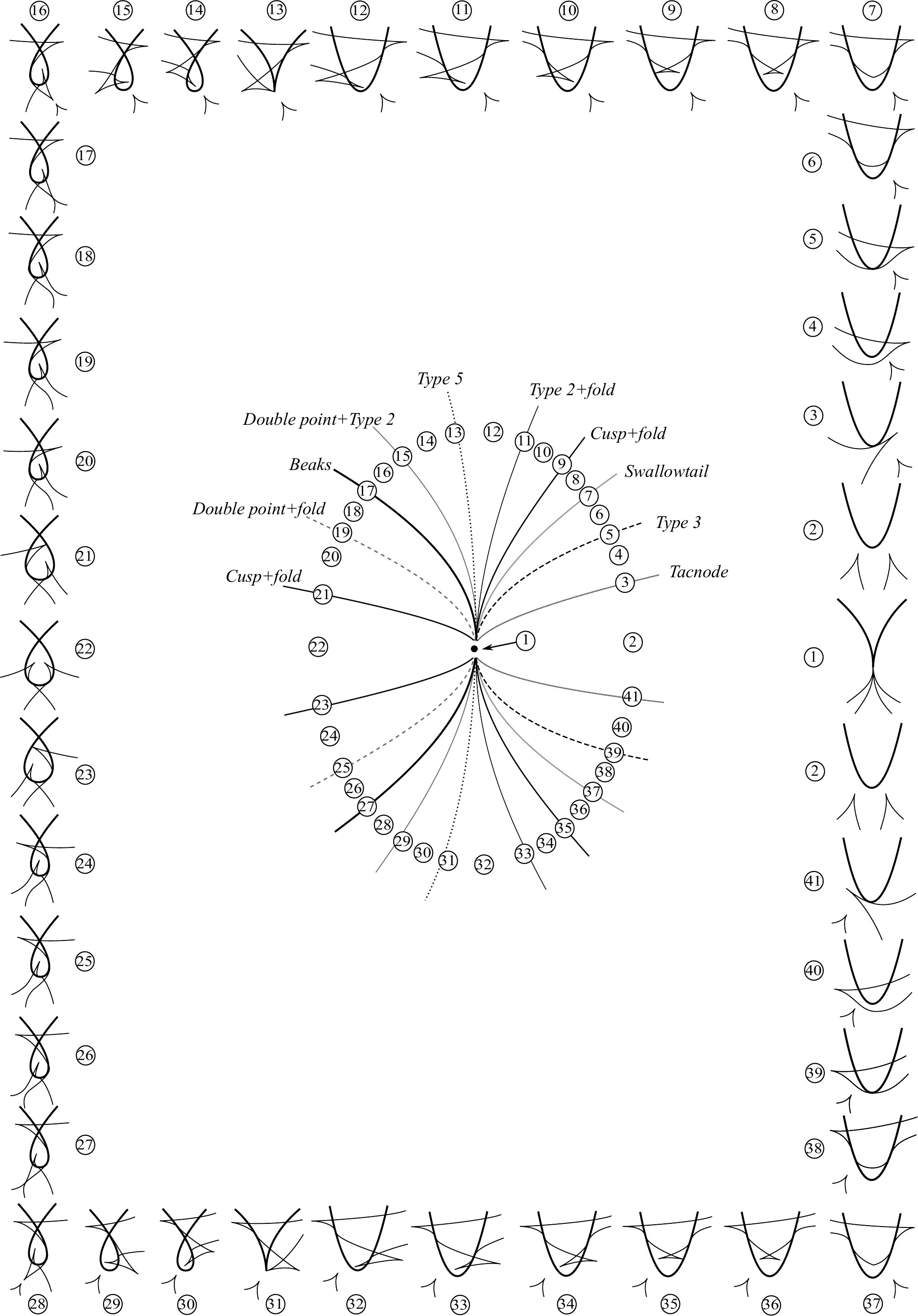}
\caption{Beaks type bifurcations of the proper contour and of the projections of the singular set at a Type 7 singularity with
$(a,b)$ in Region \textcircled{$\scriptstyle4$}, in Figure \ref{fig:Type7Strat(ab)}.}
\label{fig:ProjType7BeaksR7}
\end{center}
\end{figure}

\begin{figure}[h]
\begin{center}
\includegraphics[width=15cm, height=19cm]{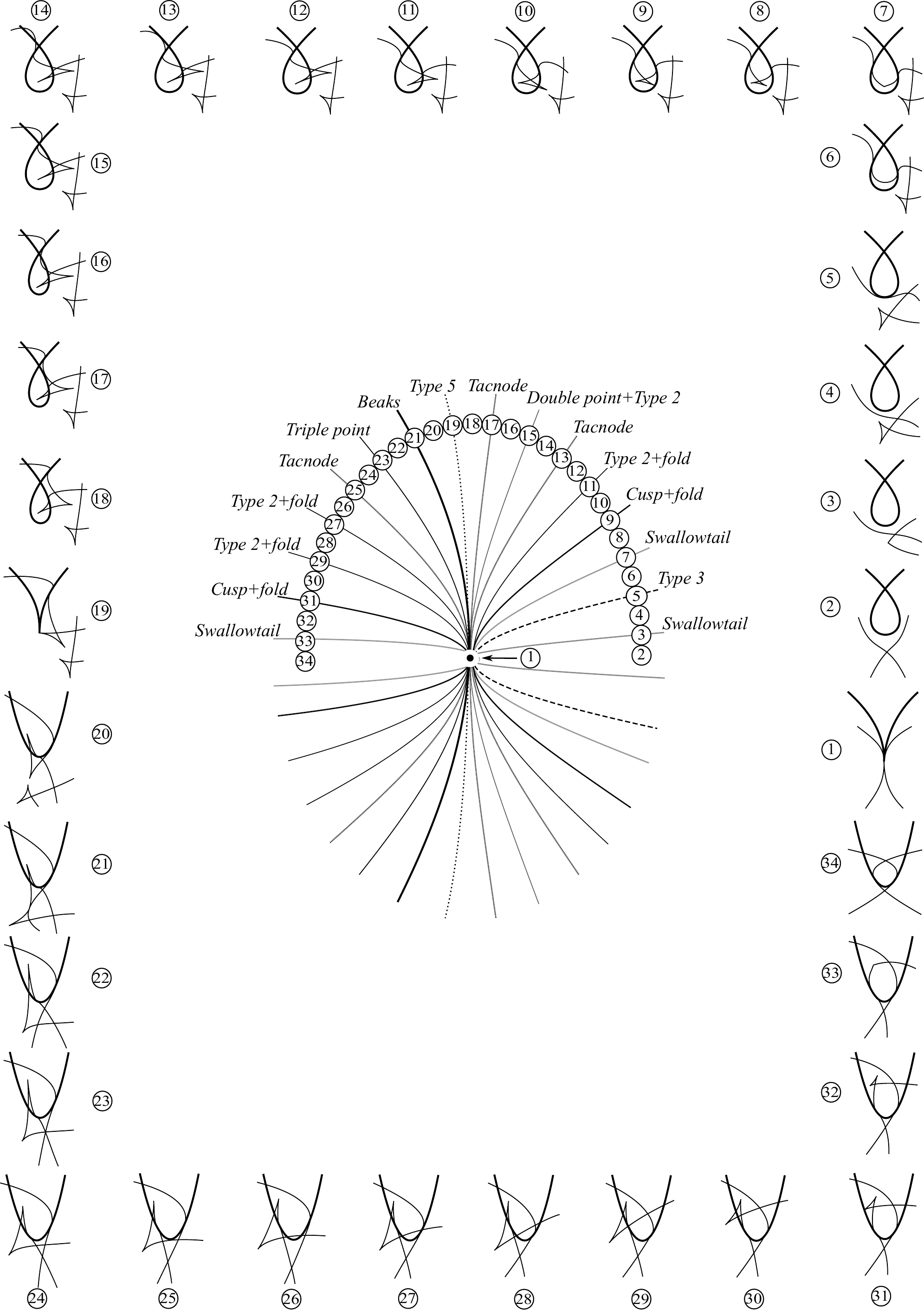}
\caption{Beaks type bifurcations of the proper contour and of the projections of the singular set at a Type 7 singularity with
$(a,b)$ in Region \textcircled{$\scriptstyle5$}, in Figure \ref{fig:Type7Strat(ab)}. The figure for $(a_1,a_2)$ with $a_2<0$ is the reflection with respect to the $y$-axis of the figure for $(a_1,-a_2)$ and is omitted for lack of space.}
\label{fig:ProjType7BeaksR6(2)}
\end{center}
\end{figure}

\bigskip

Finally, we consider geometric criteria for recognition of the generic singularities of the orthogonal projections of
the cuspidal edge $M$.
We denote by $TC_pM$ the tangent cone to $M$ at a point $p$ on the singular set $\Sigma$
and by ${\bf v}_{tg}$ the tangential direction at $p$.

\begin{prop}\label{prop:geomcondProj}
For a generic cuspidal edge, the projection  $P_{\bf v}$ can have one of the local singularities in
{\rm Table \ref{tab:germProj}} and these occur
at a point $p\in \Sigma$ when the following geometric conditions are satisfied.

{\rm (i)} If ${\bf v}$ is transverse to the tangent cone $TC_pM$, then $P_{\bf v}$ has a singularity of \mbox{\rm Type 1}.

{\rm (ii)} If $\kappa_n(p)\ne 0$, for all directions in $TC_pM\setminus \{{\bf v}_{tg}\} $ except for one ${\bf v}_{0}$, the
singularity of $P_{\bf v}$ is of \mbox{\rm Type 2}. The singularity of $P_{{\bf v}_{0}}$ is of \mbox{\rm Type 3} at generic points on $\Sigma$ and
becomes of \mbox{\rm Type 4} at isolated points on $\Sigma$. The two types \mbox{\rm Type 3} and \mbox{\rm Type 4} are distinguished by the contact
of the proper contour generator with $\Sigma$, two for \mbox{\rm Type 3} and three for \mbox{\rm Type 4}.

If $\kappa_n(p)=0$ but $\kappa_t(p)\ne 0$, the singularity of $P_{\bf v}$ is of \mbox{\rm Type 2} for
all ${\bf v}\in TC_pM\setminus \{{\bf v}_{tg}\} $. If $\kappa_n(p)=\kappa_t(p)=0$,
the singularity of $P_{\bf v}$ is of \mbox{\rm Type 3} for
all ${\bf v}\in TC_pM\setminus \{{\bf v}_{tg}\} $ except for one direction where it becomes of \mbox{\rm Type 4}.

{\rm (iii)} The singularity of  $P_{{\bf v}_{tg}}$ is of \mbox{\rm Type 5} if $\tau_\Sigma(p) \kappa_n(p)\ne 0$,
generically of \mbox{\rm Type 6} if $\tau_\Sigma(p)=0$ and $\kappa_n(p)\ne 0$, and generically of
\mbox{\rm Type 7} if $\kappa_n(p)=0$.
\end{prop}

\begin{proof}
Following the transversality theorem in the Appendix of \cite{brucewest}, for a generic cuspidal edge only singularities of
$_X\mathcal A_e$-codimension $\le 2$ (of the stratum) occur in the projection $P_{\bf v}$ (i.e., those listed in Table \ref{tab:germProj}).

The kernel direction of $(dP_{\bf v})_p$ is
parallel to the direction of projection $\bf v$.
The relative position of the kernel direction with respect to the tangent cone and the tangential direction
is invariant under diffeomorphism, so can be considered on a submersion $g$ on the model $X$.
It follows from the classification in the proof of \mbox{Theorem \ref{theo:ClassProj}}, that $g$
has a singularity of Type 1 at $p$ when the kernel direction of $dg_p$ is transverse  to the tangent cone $TC_pX$,
a singularity of Type 2,3,4 if the kernel direction is in $TC_pX$ but is not the tangential direction at $p$,
and of Type 5,6,7 if the kernel direction is parallel to the tangential direction.

We take the cuspidal edge parametrised as in \eqref{eq:prenormalform}. Suppose  that
 ${\bf v}=(\alpha,\beta,0)\in TC_pM$ with $\alpha^2+\beta^2=1$ and take $p$ to be the origin.
Consider the projection $\Pi_{\bf v}(u,v,w)=((1-\alpha^2)u-\alpha\beta v, -\alpha\beta u +(1-\beta^2) v,w)$. By rotating the
plane of projection $T_{\bf v}S^2$ to the plane $u=0$ and rescaling we get
$\Pi_{\bf v}(u,v,w)\sim_{_X\mathcal A}(\beta u-\alpha v,w)$. We shall modify the family of projections and take
$\Pi_{\bf v}(u,v,w)=(\beta u-\alpha v,w)$.

Suppose that ${\bf v}\ne {\bf v}_{tg}$ (i.e., $\beta \ne 0$; we can then set $\beta=1$). Following the proof of Theorem \ref{theo:BiffAppCont}, the singularities of Type 2, 3 and 4
are distinguished by the contact of the proper contour generator with the singular set $\Sigma$.
We have
$$
P_{\bf v}(x,y)=(\alpha x- (a(x)+\frac{1}{2}y^2), b_1(x)+y^2b_2(x)+y^3b_3(x,y)).
$$
The singular set of $P_{\bf v}$ is given by $yS_{\alpha}(x,y)=0$ with
$$
\begin{array}{rcl}
S_{\alpha}(x,y)&=&
(\alpha b_{20}+b_{12})x+\frac{1}{2}b_{03}y+\\
&&+(\frac{1}{2}\alpha b_{30}+b_{22}-\alpha a_{20} b_{12})x^2+\frac{1}{2}(b_{13}-\alpha a_{20}b_{03})xy+\frac{1}{2}\alpha b_{12}y^2+\\
&&+\alpha(\frac{2}{3} b_{40}-\frac{1}{2}a_{30}b_{12}-a_{20}b_{22})x^3-\frac{\alpha}{4}( a_{30}b_{03}+2 a_{20}b_{13})x^2y+\alpha b_{22}xy^2+\\
&&+\frac{1}{6}\alpha b_{13}y^3+O(4).
\end{array}
$$

The singular set $S_{\alpha}=0$ is transverse to  $\Sigma$ unless $\alpha b_{20}+b_{12}=0$. If
$b_{20}=\kappa_n(p)\ne 0$, then there is a unique direction ${\bf v}_{0}$ parallel to $(-b_{12},b_{20},0)$ where transversality fails (so $\alpha=\alpha_0=-b_{12}/b_{20}$). The singular set $S_{\alpha_0}=0$
has contact of  order 2 with $\Sigma$ at the origin if $2b_{12}^2a_{20}-b_{12}b_{30}+2b_{22}b_{20}\ne 0$. The contact is of order 3 if
$2b_{12}^2a_{20}-b_{12}b_{30}+2b_{22}b_{20}=0$ and
$\alpha_0(\frac{2}{3} b_{40}-\frac{1}{2}a_{30}b_{12}-a_{20}b_{22})\ne 0$.

Suppose now that $b_{20}=\kappa_n(p)=0$. Then transversality of $S_{\alpha}$ with $\Sigma$ fails if and only if
$b_{12}=\kappa_t(p)=0$. In this case the singularity of $P_{\bf v}$  is of Type 3 unless
$\alpha=\alpha_0=-2b_{22}/b_{30}$, where it becomes of Type 4.

When ${\bf v}\ne {\bf v}_{tg}=(1,0,0)$, we have
$P_{{\bf v}_{tg}}(x,y)=( a(x)+\frac{1}{2}y^2, b_1(x)+y^2b_2(x)+y^3b_3(x,y))$ and its critical set is
given by
$yS_{{\bf v}_{tg}}(x,y)=0$ with
$$
S_{{\bf v}_{tg}}(x,y)=-b_{20}x+(b_{12}a_{20}-\frac{1}{2}b_{30})x^2+\frac{1}{2}a_{20}b_{03}xy-\frac{1}{2}b_{12}y^2
+O(3).
$$

From the proof of Theorem \ref{theo:BiffAppCont}, the singularity of $P_{{\bf v}_{tg}}$ is of Type 5 (resp. Type 6)
if $S_{{\bf v}_{tg}}=0$ is a regular curve and the image of $\Sigma$ by $P_{{\bf v}_{tg}}$ is an ordinary
(resp. ramphoid) cusp. Now $\Sigma_{{\bf v}_{tg}}$ is a regular curve if and only if
$b_{20}=\kappa_n(p)\ne 0$. The image of $\Sigma$ is
$P_{{\bf v}_{tg}}(x,0)=( a(x), b_1(x))$. It has an ordinary cusp if and only if
$b_{30}a_{20}-a_{30}b_{20}\ne 0 $, that is $\tau_\Sigma(p)\ne 0$. When $\tau_\Sigma(p)=0$, the singularity
is generically a ramphoid cusp. When $b_{20}=0$,  $S_{{\bf v}_{tg}}$ is singular so we have generically a singularity of Type 7.
\end{proof}


\noindent
ROS: Departament de Geometria i Topologia, Universitat de València, c/ Dr Moliner nº 50, 46100, Burjassot, València, Spain.\\
Email: raul.oset@uv.es\\

\noindent FT: Instituto de Ci\^encias Matem\'aticas e de
Computa\c{c}\~ao - USP, Avenida Trabalhador s\~ao-carlense, 400 -
Centro,
CEP: 13566-590 - S\~ao Carlos - SP, Brazil.\\
Email: faridtari@icmc.usp.br


\begin{thebibliography}{99}

\bibitem{martinyutaro}
M. Barajas Sichaca and Y. Kabata, Projections of the cross-cap. Preprint, 2015.

\bibitem{brucegiblinBook} J. W. Bruce and P. J. Giblin,  Curves and singularities. Cambridge University Press, 1992.

\bibitem{brucegiblinProjBound}
J. W. Bruce and  P. J. Giblin, Projections of surfaces with boundary.
\emph{Proc. London Math. Soc.} 60 (1990), 392--416.


\bibitem{bgt95}J. W. Bruce, P. J. Giblin  and F. Tari, Families of surfaces: height functions, Gauss maps and duals.
{\it Pitman Res. Notes Math. Ser.} 333 (1995), 148--178.

\bibitem{bkd} J. W. Bruce, N. P. Kirk and A. A. du Plessis, Complete transversals
and the classification of singularities. {\it Nonlinearity} 10 (1997), 253-275.

\bibitem{bdw}
J. W. Bruce, A. A. du Plessis  and C. T. C. Wall,
Determinacy and unipotency.
\emph{Invent. Math.} 88 (1987), 521--554.

\bibitem{bruceroberts}J. W. Bruce and R. M. Roberts, Critical points of functions on analytic varieties.
{\it Topology 27} (1988), no. 1, 57–-90.

\bibitem{brucewest}J. W. Bruce and J. M. West, Functions on a crosscap.
{\it Math. Proc. Cambridge Philos. Soc}. 123 (1998), 19--39.

\bibitem{bruce-wilkinson} J. W. Bruce and T. C. Wilkinson, Folding maps and focal sets.  Lecture Notes in Math., 1462, 63--72, Springer, Berlin, 1991.


\bibitem{damon}
J. N. Damon, Topological triviality and versality for subgroups of
\mbox{${\mathcal A}$} and \mbox{${\mathcal K}$}. \emph{Mem. Amer. Math. Soc.}
{75} (1988),  no. 389.

\bibitem{gaffney}
T. Gaffney, The structure of the $T\mathcal A(f)$, classification and application to differential geometry.
{Singularities, Part 1 (Arcata, Calif., 1981)}. {\it Proc. Sympos. Pure Math.} 40 (1983) 09--427.
 Amer. Math. Soc., Providence, RI.

\bibitem{GibsonHobbs} C. G. Gibson and C. A.  Hobbs, Singularities of general one-dimensional motions of the plane and space.
{\it Proc. Roy. Soc. Edinburgh Sect. A}  125  (1995),  639--656.


\bibitem{ShCarCidFaBook}
S. Izumiya, M. C. Romero Fuster, M. A. S. Ruas and F. Tari, \emph{Differential geometry from singularity theory viewpoint}.
World Scientific, 2015.

\bibitem{transkirk}
N. P. Kirk, {Computational aspects of classifying
singularities}. {\it LMS J. Comput. Math.} 3 (2000), 207–228 .

\bibitem{koenderink}
J. J. Koenderink, \emph{Solide Shape}. MIT Press,
Cambridge, MA, 1990.

\bibitem{martinsnuno} L. Martins and J. J. Nu\~no Ballesteros,
Contact properties of surfaces in $\mathbb R^3$ with corank 1 singularities. {\it Tohoku Math. J.}
 67 (2015), 105--124.


\bibitem{martinssaji} L. Martins and K. Saji, Geometric invariants
of cuspidal edges. To appear in {\it Canad. J. Math.}


\bibitem{rieger} J. H. Rieger, Families of maps from the plane to the plane. {\it J. London Math. Soc.} 36 (1987), 351--369.

\bibitem {riegerruas} J. H. Rieger and M. A. S. Ruas,
Classification of $\mathcal A$-simple germs from $\mathbb K^n$
to $\mathbb K^2$.  {\it Compositio Math.} 79 (1991), 99--108.

\bibitem{kentaro} Saji, Kentaro; Umehara, Masaaki; Yamada, Kotaro, The geometry of fronts. {\it Ann. of Math.} (2) 169 (2009),  491--529.

\bibitem{taricreases}
F. Tari, Projections of piecewise-smooth surfaces.
\emph{J. London Math. Soc.} 44 (1991), 155--172.

\bibitem{teramoto} K. Teramoto, Parallel and dual surfaces of cuspidal edges. Preprint.

\bibitem{wallgenericgeometry} C. T. C. Wall,
Geometric properties of generic differentiable manifolds.
{\it Lecture Notes in Math.}, 597 (1977), 707--774.

\bibitem{west} J. M. West, The differential geometry of the crosscap.  Ph.D. thesis, The University of Liverpool, 1995.

\bibitem{wilkinson} T. C. Wilkinson, The geometry of folding maps. PhD Thesis, University of Newcastle-upon-Tyne, 1991.

\end{thebibliography}
\end{document}